\documentclass{amsart}
\usepackage{amsthm,xcolor, bm}
\usepackage{latexsym,amssymb,amsfonts,amsmath,graphicx, url,mathtools}
\usepackage[vcentermath,enableskew]{youngtab}
\usepackage{color}
\usepackage[noadjust]{cite} 
\usepackage{tikz,hyperref, mathrsfs}
\usepackage[margin=1in]{geometry}
\usepackage{enumitem}

\newtheorem{theorem}{Theorem}[section]
\newtheorem{proposition}[theorem]{Proposition}
\newtheorem{lemma}[theorem]{Lemma}

\newtheorem{corollary}[theorem]{Corollary}

\theoremstyle{definition}
\newtheorem{definition}[theorem]{Definition}
\newtheorem{example}[theorem]{Example}

\newtheoremstyle{named}{}{}{\itshape}{}{\bfseries}{.}{.5em}{#1 \thmnote{#3}}
\theoremstyle{named}
\newtheorem*{namedtheorem}{Theorem}

\def\Z{\mathbb{Z}}

\def\hd{{\widehat{D}}}
\def\hc{{\widehat{C}}}

\title{Zero-one Schubert polynomials}
\author{Alex Fink}
\address{Alex Fink, School of Mathematical Sciences, Queen Mary University of London, UK, E1 4NS.
	\textup{a.fink@qmul.ac.uk}}

\author{Karola M\'esz\'aros}
\address{Karola M\'esz\'aros, Department of Mathematics, Cornell University, Ithaca, NY 14853 and School of Mathematics, Institute for Advanced Study, Princeton, NJ 08540.  \newline\textup{karola@math.cornell.edu}
}

\author{Avery St.~Dizier}
\address{Avery St. Dizier, Department of Mathematics, Cornell University, Ithaca NY 14853.  \newline\textup{ajs624@cornell.edu}
}

\thanks{Fink was partially supported by an Engineering and Physical Sciences Research Council grant (EP/M01245X/1),
	and is now supported by the European Union's Horizon 2020 research and innovation programme under grant agreement No~792432.
	M\'esz\'aros is partially supported by a National Science Foundation Grant (DMS 1501059), as well as by a von Neumann Fellowship at the IAS funded by the Fund for Mathematics and Friends of the Institute for Advanced Study.}

\date{\today}
\begin{document}

\begin{abstract}
	We prove that if $\sigma \in S_m$ is a pattern of $w \in S_n$, then we can express the Schubert polynomial $\mathfrak{S}_w$ as a monomial times $\mathfrak{S}_\sigma$ (in reindexed variables) plus a polynomial with nonnegative coefficients. This implies that the set of permutations whose Schubert polynomials have all their coefficients equal to either 0 or 1 is closed under pattern containment. Using Magyar's orthodontia, we characterize this class by a list of twelve avoided patterns. We also give other equivalent conditions on $\mathfrak{S}_w$ being zero-one. In this case, the Schubert polynomial $\mathfrak{S}_w$ is equal to the integer point transform of a generalized permutahedron. 
\end{abstract}

\maketitle 

\section{Introduction} 
	Schubert polynomials, introduced  by Lascoux and Sch\"utzenberger in \cite{LS}, represent cohomology classes of Schubert cycles in the flag variety. Knutson and Miller also showed them to be multidegrees of matrix Schubert varieties \cite{multidegree}. There are a number of combinatorial formulas for the Schubert polynomials \cite{laddermoves, BJS, FK1993, nilcoxeter, thomas, lenart, manivel, prismtableaux}, 
	yet only recently has the structure of their supports been investigated: 
	the support of a Schubert polynomial $\mathfrak{S}_w$ is the set of all integer points of a certain generalized permutahedron $\mathcal{P}(w)$ \cite{FMS, MTY}.
	The question motivating this paper is to characterize when  $\mathfrak{S}_w$ equals the integer point transform of $\mathcal{P}(w)$, 
	in other words, when all the coefficients of $\mathfrak{S}_w$ are equal to $0$ or $1$.

	One of our main results is a pattern-avoidance characterization of the permutations corresponding to these polynomials:

	\begin{theorem} \label{thm:01}
		The Schubert polynomial $\mathfrak{S}_w$ is zero-one if and only if $w$ avoids the patterns $12543$, $13254$, $13524$, $13542$, $21543$, $125364$, $125634$, $215364$, $215634$, $315264$, $315624$, and $315642$.
	\end{theorem}
	
	In Theorem~\ref{thm:011} we provide further equivalent conditions on the Schubert polynomial $\mathfrak{S}_w$ being zero-one. One implication of Theorem~\ref{thm:01} follows from our other main result, 
which relates the Schubert polynomials $\mathfrak{S}_\sigma$ and $\mathfrak{S}_w$ when $\sigma$ is a pattern of $w$:
	
	\begin{theorem} 
		\label{thm:pattern}  
		Fix  $w \in S_n$ and let $\sigma \in S_{n-1}$ be the pattern with Rothe diagram $D(\sigma)$ obtained by removing row $k$  and column $w_k$ from $D(w)$. Then
		\begin{align*}
			\mathfrak{S}_{w}(x_1, \ldots, x_n)=M(x_1, \ldots, x_n) \mathfrak{S}_{\sigma}(x_{1}, \ldots, \widehat{x_k}, \ldots, x_{n})+F(x_1, \ldots, x_n),
		\end{align*}  
		where $F\in \Z_{\geq 0}[x_1, \ldots, x_n]$ and  
		\[M(x_1, \ldots, x_n) = \left(\prod_{(k,i)\in D(w)}{x_k}\right)\left(\prod_{(i,w_k)\in D(w)}{x_i} \right).\]		
	\end{theorem}

	Theorem~\ref{thm:pattern} is a special case of Theorem~\ref{thm:pattern2}, which applies to the dual character of the flagged Weyl module of any diagram.
	
	\subsection*{Outline of this paper}
	Section~\ref{sec:magyar} gives an expression of Magyar for Schubert polynomials in terms of orthodontic sequences $(\bm{i}, \bm{m})$. In Section~\ref{sec:suff}, we give a condition ``multiplicity-free'' on the orthodontic sequence $(\bm{i}, \bm{m})$ of $w$  which implies that $\mathfrak{S}_w$ is zero-one. In Section~\ref{sec:mult-patt} we show that multiplicity-freeness can equivalently be phrased in terms of pattern avoidance. We then prove in Section~\ref{sec:mult-patt} that multiplicity-freeness is also a necessary condition for $\mathfrak{S}_w$ to be zero-one. In the latter proof we assume Theorem~\ref{thm:pattern}, whose generalization (Theorem~\ref{thm:pattern2})  and proof is the subject of Section~\ref{sec:trans}.
	
\section{Magyar's orthodontia for Schubert polynomials}
	\label{sec:magyar}

	In this section we explain the results we use to show one direction of Theorem~\ref{thm:01}. We include the classical definition of Schubert polynomials here for reference. 
	
	The \emph{Schubert polynomial} of the longest permutation $w_0=n \hspace{.1cm} n\!-\!1 \hspace{.1cm} \cdots \hspace{.1cm} 2 \hspace{.1cm} 1 \in S_n$ is 
	\[\mathfrak{S}_{w_0}\coloneqq x_1^{n-1}x_2^{n-2}\cdots x_{n-1}.\]
	
	For $w\in S_n$, $w\neq w_0$, there exists $i\in [n-1]$ such that $w_i<w_{i+1}$. 
	For any such~$i$, the Schubert polynomial $\mathfrak{S}_{w}$ is defined as 
	\[\mathfrak{S}_{w}(x_1, \ldots, x_n)\coloneqq \partial_i \mathfrak{S}_{ws_i}(x_1, \ldots, x_n),\] 
	where $s_i$ is the transposition swapping $i$ and $i+1$, and $\partial_i$ is the $i$th divided difference operator
	\[\partial_i (f):=\frac{f(x_1,\ldots,x_n)-f(x_1,\ldots,x_{i-1},x_{i+1},x_i,x_{i+2},\ldots,x_n)}{x_i-x_{i+1}}.\]
	Since the operators $\partial_i$ satisfy the braid relations, the Schubert polynomials $\mathfrak{S}_{w}$ are well-defined. 
	
	We will not be using the above definition of Schubert polynomials in this work. Instead, we will make use of several results due to Magyar in \cite{magyar}. We start by summarizing Proposition 15 and Proposition 16 of \cite{magyar} and supplying the necessary background, closely following the exposition of \cite{magyar}.
	
	By a \emph{diagram}, we mean a sequence $D=(C_1,C_2,\ldots,C_n)$ of finite subsets of $[n]$, called the \emph{columns} of $D$. We interchangeably think of $D$ as a collection of boxes $(i,j)$ in a grid, viewing an element $i\in C_j$ as a box in row $i$ and column $j$ of the grid. When we draw diagrams, we read the indices as in a matrix: $i$ increases top-to-bottom and $j$ increases left-to-right. 
	Two diagrams $D$ and $D'$ are called \emph{column-equivalent} if one is obtained from the other by reordering nonempty columns and adding or removing any number of empty columns. For a column $C\subseteq [n]$, let the \emph{multiplicity} $\mathrm{mult}_D(C)$ be the number of columns of $D$ which are equal to $C$. The sum of diagrams, denoted $D\oplus D'$, is constructed by concatenating the lists of columns; graphically this means placing $D'$ to the right of $D$. 
	
	The \emph{Rothe diagram} $D(w)$ of a permutation $w\in S_n$ is the diagram
	\[ D(w)=\{(i,j)\in [n]\times [n] \mid i<(w^{-1})_j\mbox{ and } j<w_i \}. \]
	Note that Rothe diagrams have the \emph{northwest property}: If $(r,c'),(r',c)\in D(w)$ with $r<r'$ and $c<c'$, then $(r,c)\in D(w)$.
	
	\begin{example}
		If $w=31542$, then 
	\begin{center}
		\begin{tikzpicture}[scale=.55]
			\draw (0,0)--(5,0)--(5,5)--(0,5)--(0,0);
			
			\filldraw[draw=black,fill=lightgray] (0,4)--(1,4)--(1,5)--(0,5)--(0,4);
			\filldraw[draw=black,fill=lightgray] (1,4)--(2,4)--(2,5)--(1,5)--(1,4);
			\filldraw[draw=black,fill=lightgray] (1,2)--(2,2)--(2,3)--(1,3)--(1,2);
			\filldraw[draw=black,fill=lightgray] (1,1)--(2,1)--(2,2)--(1,2)--(1,1);
			\filldraw[draw=black,fill=lightgray] (3,2)--(4,2)--(4,3)--(3,3)--(3,2);
			
			\node at (-1.5,2.5) {$D(w)=$};
			
			\node at (8.9,2.5) {$=(\{1\},\{1,3,4\},\emptyset,\{3\},\emptyset).$};
		\end{tikzpicture}
	\end{center}
	\end{example}
	
	We next recall Magyar's orthodontia. Let $D$ be the Rothe diagram of a permutation $w\in S_n$ with columns $C_1,C_2,\ldots,C_n$. We describe an algorithm for constructing a reduced word $\bm{i}=(i_1,\ldots,i_l)$ and a multiplicity list $\bm{m}=(k_1,\ldots,k_n;\,m_1,\ldots,m_l)$ such that the diagram $D_{\bm{i},\bm{m}}$ defined by
	\[D_{\bm{i},\bm{m}} = \bigoplus_{j=1}^n {k_j\cdot [j]} \quad \oplus \quad \bigoplus_{j=1}^l m_j\cdot (s_{i_1}s_{i_2}\cdots s_{i_j}[i_j]), \]
	is column-equivalent to~$D$. In the above, $m\cdot C$ denotes $C\oplus \cdots\oplus C$ with $m$ copies of~$C$; in particular $0\cdot C$ should be interpreted as a diagram with no columns, not the empty column.
	
	The algorithm to produce $\bm{i}$ and $\bm{m}$ from $D$ is as follows. To begin the first step, for each $j\in [n]$ let $k_j=\mathrm{mult}_D([j])$, the number of columns of $D$ of the form $[j]$. Replace all such columns by empty columns for each $j$ to get a new diagram $D_-$.
	
	Given a column $C\subseteq [n]$, a \emph{missing tooth} of $C$ is a positive integer $i$ such that $i\notin C$, but $i+1\in C$. The only columns without missing teeth are the empty column and the intervals $[i]$. Hence the first nonempty column of $D_-$ (if there is any) contains a smallest missing tooth $i_1$. Switch rows $i_1$ and $i_1+1$ of $D_-$ to get a new diagram $D'$.
	
	In the second step, repeat the above with $D'$ in place of $D$. That is, let $m_1=\mathrm{mult}_{D'}([i_1])$ and replace all columns of the form $[i_1]$ in $D'$ by empty columns to get a new diagram $D_-'$. Find the smallest missing tooth $i_2$ of the first nonempty column of $D_-'$, and switch rows $i_2$ and $i_2+1$ of $D_-'$ to get a new diagram $D''$.
	
	Continue in this fashion until no nonempty columns remain. It is easily seen that the sequences $\bm{i}=(i_1,\ldots,i_l)$ and $\bm{m}=(k_1,\ldots,k_n;\,m_1,\ldots,m_l)$ just constructed have the desired properties.
	
	\begin{definition}
		\label{def:imsequence}
		The pair $(\bm{i},\bm{m})$ constructed from the preceding algorithm is called the \emph{orthodontic sequence} of $w$.
	\end{definition}
	
	\begin{example}
		If $w=31542$, then the orthodontic sequence algorithm produces the diagrams
		\begin{center}
			\includegraphics[scale=.85]{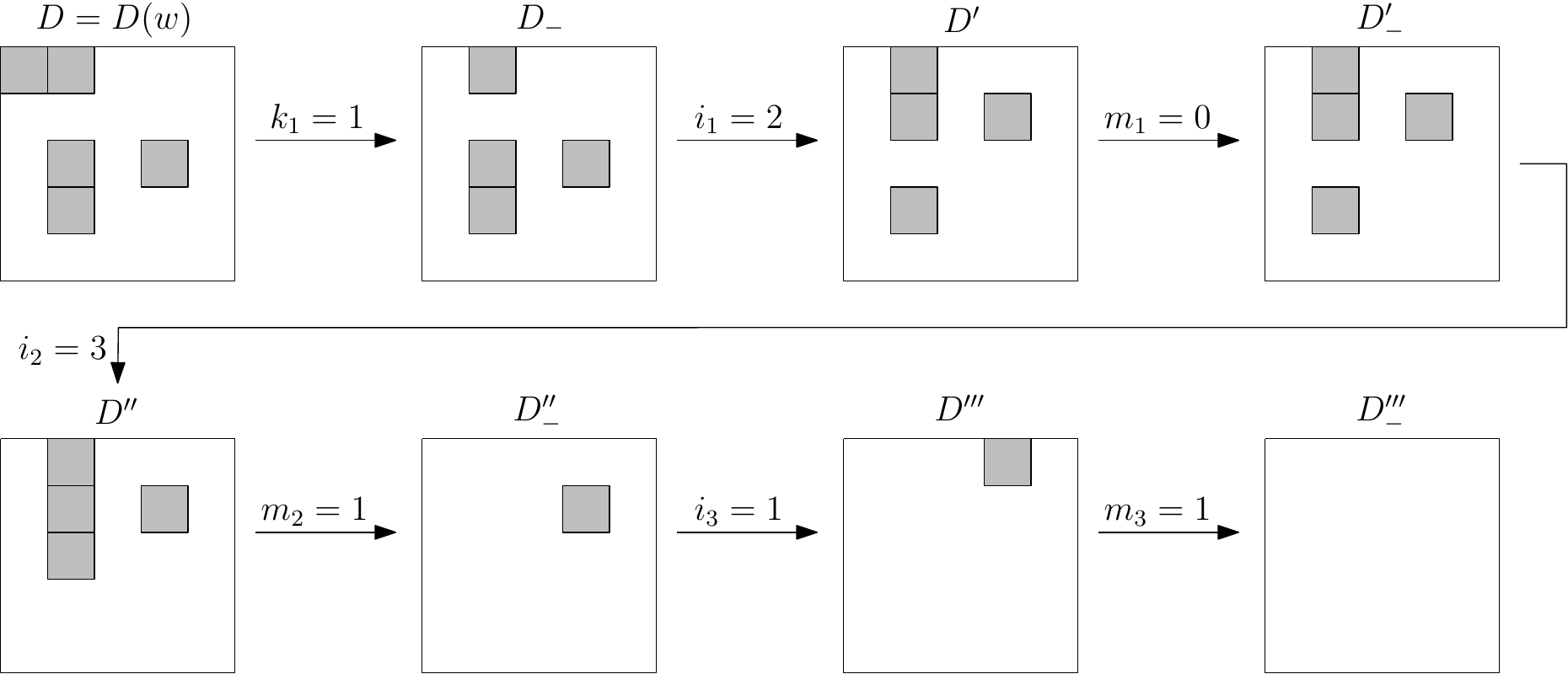}
		\end{center}
		The sequence of missing teeth gives $\bm{i}=(2,3,1)$ and $\bm{m}=(1,0,0,0,0;\,0,1,1)$, so
		\begin{center}
			\begin{tikzpicture}[scale=.5]
				\draw (9.5,0)--(12.5,0)--(12.5,5)--(9.5,5)--(9.5,0);
				\filldraw[draw=black,fill=lightgray] (9.5,4)--(10.5,4)--(10.5,5)--(9.5,5)--(9.5,4);
				\filldraw[draw=black,fill=lightgray] (10.5,4)--(11.5,4)--(11.5,5)--(10.5,5)--(10.5,4);
				\filldraw[draw=black,fill=lightgray] (10.5,2)--(11.5,2)--(11.5,3)--(10.5,3)--(10.5,2);
				\filldraw[draw=black,fill=lightgray] (10.5,1)--(11.5,1)--(11.5,2)--(10.5,2)--(10.5,1);
				\filldraw[draw=black,fill=lightgray] (11.5,2)--(12.5,2)--(12.5,3)--(11.5,3)--(11.5,2);
				\node at (8,2.5) {$D_{\bm{i},\bm{m}}=$};
				\node at (12.7,2.5) {.};
			\end{tikzpicture}
		\end{center}
	\end{example}

	\begin{theorem}[{\cite[Proposition~15]{magyar}}]
		\label{thm:magyaroperatortheorem}
		Let $w\in S_n$ have orthodontic sequence $(\bm{i},\bm{m})$. If $\pi_j=\partial_j x_j$ denotes the $j$th Demazure operator and $\omega_j=x_1x_2\cdots x_j$, then
		\[\mathfrak{S}_w = \omega_1^{k_1}\cdots\omega_n^{k_n}\pi_{i_1}(\omega_{i_1}^{m_1} \pi_{i_2}(\omega_{i_2}^{m_2}\cdots \pi_{i_l}(\omega_{i_l}^{m_l})\cdots )). \]
	\end{theorem}
	
	\begin{example}
		For $w=31542$, it is easily checked that
		\[\mathfrak{S}_w=x_1\pi_2\pi_3(x_1x_2x_3\pi_1(x_1)). \]
	\end{example}
	
	Theorem~\ref{thm:magyaroperatortheorem} can also be realized on the level of tableaux, analogous to Young tableaux in the case of Schur polynomials. A \emph{filling} (with entries in $\{1,...,n\}$) of a diagram $D$ is a map $T$ assigning to each box in $D$ an integer in $[n]$. A filling $T$ is called \emph{column-strict} if $T$ is strictly increasing down each column of $D$. The \emph{weight} of a filling $T$ is the vector $wt(T)$ whose $i$th component $wt(T)_i$ is the number of times $i$ occurs in $T$. 

	Given a permutation $w\in S_n$ with orthodontic sequence $(\bm{i},\bm{m})$, we will define a set $\mathcal{T}_w$ of fillings of the diagram $D_{\bm{i},\bm{m}}$ which satisfy 
	\[\mathfrak{S}_w=\sum_{T\in\mathcal{T}_w}x_1^{wt(T)_1}x_2^{wt(T)_2}\cdots x_n^{wt(T)_n}. \]
	We start by recalling the \emph{root operators}, first defined in \cite{rootoperators}. These are operators $f_i$ which either take a filling $T$ of a diagram $D$ to another filling of $D$, or are undefined on $T$. To define root operators, we first encode a filling $T$ in terms of its reading word. The \emph{reading word} of a filling $T$ of a diagram $D=D_{\bm{i},\bm{m}}$ is the sequence of the entries of $T$ read in order, down each column going left-to-right along columns; that is the sequence 
	\[T(1,1),T(2,1),\ldots,T(n,1),T(1,2),T(2,2),\ldots,T(n,2),\ldots, T(n,n) \]
	ignoring any boxes $(i,j)\notin D$.
	
	If it is defined, the operator $f_i$ changes an entry of $i$ in $T$ to an entry of $i+1$ according to the following rule. First, ignore all the entries in $T$ except those which equal $i$ or $i+1$. Now ``match parentheses'':
	if, in the list of entries not yet ignored, an $i$ is followed by an $i+1$, then henceforth ignore that pair of entries as well; look again for an $i$ followed (up to ignored entries) by an $i+1$, and henceforth ignore this pair; continue doing this until all no such pairs remain unignored. The remaining entries of $T$ will be a subword of the form $i+1,i+1,\ldots,i+1,i,i,\ldots,i$. If $i$ does not appear in this word, then $f_i(T)$ is undefined. Otherwise, $f_i$ changes the leftmost $i$ to an $i+1$. Reading the image word back into $D$ produces a new filling. We can iteratively apply $f_i$ to a filling $T$. 
	
	\begin{example}
		If $T=3\,1\,2\,2\,2\,1\,3\,1\,2\,4\,3\,2\,4\,1\,3\,1$, applying $f_1$ iteratively to $T$ yields:
		\begin{center}
			\begin{tabular}{rccccccccccccccccc}
				$\phantom{f_1(}T\phantom{)}$=&3&1&2&2&2&1&3&1&2&4&3&2&4&1&3&1&\\
				$\phantom=$&$\cdot$&1&2&2&2&1&$\cdot$&1&2&$\cdot$&$\cdot$&2&$\cdot$&1&$\cdot$&1&\\
				$\phantom=$&$\cdot$&$\cdot$&$\cdot$&2&2&1&$\cdot$&$\cdot$&$\cdot$&$\cdot$&$\cdot$&2&$\cdot$&1&$\cdot$&1&\\
				$\phantom=$&$\cdot$&$\cdot$&$\cdot$&2&2&$\cdot$&$\cdot$&$\cdot$&$\cdot$&$\cdot$&$\cdot$&$\cdot$&$\cdot$&1&$\cdot$&1&\\
				$f_1(T)=$&3&1&2&2&2&1&3&1&2&4&3&2&4&\textbf{2}&3&1&\\
				$f_1^2(T)=$&3&1&2&2&2&1&3&1&2&4&3&2&4&2&3&\textbf{2}&\\
				$f_1^3(T)\mathrel{\phantom=}$&&&&&&&&&&&&&&&&&\hspace{-45ex}\mbox{is undefined}
			\end{tabular}
		\end{center}
	\end{example}
	
	Define the set-valued \emph{quantized Demazure operator} $\widetilde{\pi}_i$ by $\widetilde{\pi}_i(T)=\{T,f_i(t),f_i^2(T),\ldots\}$. For a set $\mathcal{T}$ of tableaux, let 
	\[\widetilde{\pi}_{i}(\mathcal{T})=\bigcup_{T\in\mathcal{T}}\widetilde{\pi}_i(T).\]
	Next, consider the column $[j]$ and its minimal column-strict filling $\widetilde{\omega}_j$ ($j$th row maps to $j$). For a filling $T$ of a diagram $D$ with columns $(C_1,C_2,\ldots,C_n)$, define in the  obvious way the composite filling of $[j]\oplus D$, corresponding to concatenating the reading words of $[j]$ and $D$. Define $[j]^r\oplus D$ analogously by adding $r$ columns $[j]$ to $D$, each with filling $\widetilde{\omega}_j$.
	
	\begin{definition}
		\label{def:magyartableauxset}
		Let $w\in S_n$ be a permutation with orthodontic sequence $(\bm{i},\bm{m})$. Define the set $\mathcal{T}_w$ of tableaux by
		\[\mathcal{T}_w=\widetilde{\omega}_1^{k_1}\oplus\cdots\oplus\widetilde{\omega}_n^{k_n}\oplus\widetilde{\pi}_{i_1}(\widetilde{\omega}_{i_1}^{m_1}\oplus \widetilde{\pi}_{i_2}(\widetilde{\omega}_{i_2}^{m_2}\oplus\cdots \oplus\widetilde{\pi}_{i_l}(\widetilde{\omega}_{i_l}^{m_l})\cdots )). \]
	\end{definition}

	\begin{theorem}[{\cite[Proposition~16]{magyar}}]
		\label{thm:magyartableauxtheorem}
		Let $w\in S_n$ be a permutation with orthodontic sequence $(\bm{i},\bm{m})$. Then,
		\[\mathfrak{S}_w=\sum_{T\in\mathcal{T}_w}x_1^{wt(T)_1}x_2^{wt(T)_2}\cdots x_n^{wt(T)_n}. \]
	\end{theorem}
	
	\begin{example}
		Consider again $w=31542$, so the orthodontic sequence of $w$ is $\bm{i}=(2,3,1)$ and $\bm{m}=(1,0,0,0,0; 0,1,1)$.
		The set $\mathcal{T}_w$ is built up as follows:
		\begin{align*}
			\{\}&\xrightarrow{\widetilde{\omega}_1}\{1\}\xrightarrow{\widetilde{\pi}_1}\{1,2\}\xrightarrow{\widetilde{\omega}_3}\{1231,1232\}\xrightarrow{\widetilde{\pi}_3}\{1231,1241,1232,1242\}\\
			&\xrightarrow{\widetilde{\pi}_2}\{1231,1241,1341,1232,1233,1242,1342,1343\}\\
			&\xrightarrow{\widetilde{\omega}_1}\{11231,11241,11341,11232,11233,11242,11342,11343\}
		\end{align*}
		which agrees with
		\[\mathfrak{S}_{w}=x_1^3x_2x_3+x_1^3x_2x_4+x_1^3x_3x_4+x_1^2x_2^2x_3+x_1^2x_2x_3^2+x_1^2x_2^2x_4+x_1^2x_2x_3x_4+x_1^2x_3^2x_4.\]
	\end{example}

	We now describe a way to view each step of the construction of $\mathcal{T}_w$ as producing a set of fillings of a diagram.
	\begin{definition}
		\label{def:partialfilling}
		Let $w$ be a permutation with orthodontic sequence $(\bm{i},\bm{m})$, $\bm{i}=(i_1,\ldots,i_l)$. For each $r\in [l]$, define 
		\[\mathcal{T}_w(r)=\widetilde{\omega}_{i_r}^{m_r}\oplus\widetilde{\pi}_{i_{r+1}}(\widetilde{\omega}_{i_{r+1}}^{m_{r+1}}\oplus\cdots\oplus \widetilde{\pi}_{i_l}(\widetilde{\omega}_{i_l})\cdots). \]
		Set $\mathcal{T}_w(0)=\mathcal{T}_w$.
	\end{definition}
	\begin{definition}
		\label{def:partialdiagram}
		Let $w$ be a permutation with orthodontic sequence $(\bm{i},\bm{m})$, $\bm{i}=(i_1,\ldots,i_l)$. For any $r \in [l]$, let $O(w,r)$ be the diagram obtained from $D(w)$ in the construction of $(\bm{i},\bm{m})$ at the time when the row swaps of the missing teeth $i_1,\ldots,i_{r}$ have all been executed on $D(w)$, but after executing the row swap of the missing tooth $i_r$, columns without missing teeth have not yet been removed ($m_r$ has not yet been recorded). Set $O(w,0)=D(w)$. For each $r$, give $O(w,r)$ the same column indexing as $D(w)$, so any columns replaced by empty columns in the execution of the missing teeth $i_1,\ldots,i_{r-1}$ retain their original index in $D(w)$.
	\end{definition}

	The motivation behind Definition~\ref{def:partialfilling} and Definition~\ref{def:partialdiagram} is that the elements of $\mathcal{T}_w(r)$ can be viewed as column-strict fillings of $O(w,r)$ for each $r$. To do this, the choice of filling order for $O(w,r)$ is crucial. Let $w\in S_n$ and consider $D=D(w)$ and $D_{\bm{i},\bm{m}}$. Suppose $D$ has $z$ nonempty columns. There is a unique permutation $\tau$ of $[n]$ taking the column indices of $D$ to the column indices of $D_{\bm{i},\bm{m}}\oplus \emptyset^{n-z}$ with the following properties:
	\begin{itemize}
		\item Column $c$ of $D$ is the same as column $\tau(c)$ of $D_{\bm{i},\bm{m}}$.
		\item If column $c$ and column $c'$ of $D$ are equal with $c<c'$, then $\tau(c)<\tau(c')$.
	\end{itemize} 
	Recall that the columns of $O(w,r)$ have the same column labels as $D$. To read an element $T\in \mathcal{T}_w(r)$ into $O(w,r)$, read $T$ left-to-right and fill in top-to-bottom columns $\tau^{-1}(n),\,\tau^{-1}(n-1),\ldots,\tau^{-1}(1)$ (ignoring any column indices referring to empty columns).
	
	\begin{lemma}
		\label{lem:fillings}
		Let $w\in S_n$ have orthodontic sequence $(\bm{i},\bm{m})$, $\bm{i}=(i_1,\ldots,i_l)$. In the filling order specified above, the elements of $\mathcal{T}_w(r)$ are column-strict fillings of $O(w,r)$ for each $0\leq r\leq l$.
	\end{lemma}

	\begin{example}
		Take again $w=31542$ with orthodontic sequence $\bm{i}=(2,3,1)$ and $\bm{m}=(1,0,0,0,0;0,1,1)$.
		Recall that 
		\begin{center}
			\begin{tikzpicture}[scale=.5]
				\draw (0,0)--(5,0)--(5,5)--(0,5)--(0,0);
				\filldraw[draw=black,fill=lightgray] (0,4)--(1,4)--(1,5)--(0,5)--(0,4);
				\filldraw[draw=black,fill=lightgray] (1,4)--(2,4)--(2,5)--(1,5)--(1,4);
				\filldraw[draw=black,fill=lightgray] (1,2)--(2,2)--(2,3)--(1,3)--(1,2);
				\filldraw[draw=black,fill=lightgray] (1,1)--(2,1)--(2,2)--(1,2)--(1,1);
				\filldraw[draw=black,fill=lightgray] (3,2)--(4,2)--(4,3)--(3,3)--(3,2);
				\node at (-1.6,2.5) {$D(w)=$};
				\node at (6,2.5) {and};
				
				\draw (9.5,0)--(12.5,0)--(12.5,5)--(9.5,5)--(9.5,0);
				\filldraw[draw=black,fill=lightgray] (9.5,4)--(10.5,4)--(10.5,5)--(9.5,5)--(9.5,4);
				\filldraw[draw=black,fill=lightgray] (10.5,4)--(11.5,4)--(11.5,5)--(10.5,5)--(10.5,4);
				\filldraw[draw=black,fill=lightgray] (10.5,2)--(11.5,2)--(11.5,3)--(10.5,3)--(10.5,2);
				\filldraw[draw=black,fill=lightgray] (10.5,1)--(11.5,1)--(11.5,2)--(10.5,2)--(10.5,1);
				\filldraw[draw=black,fill=lightgray] (11.5,2)--(12.5,2)--(12.5,3)--(11.5,3)--(11.5,2);
				\node at (8,2.5) {$D_{\bm{i},\bm{m}}=$};
				\node at (12.7,2.5) {,};
				
			\end{tikzpicture}
		\end{center}
		so $\tau=12435=\tau^{-1}$. Consider the elements $1\in\mathcal{T}_w(3)$, $1232\in\mathcal{T}_w(2)$, $1242\in\mathcal{T}_w(1)$, and $11342\in\mathcal{T}_w(0)$. The column filling order of each $O(w,r)$ is given by reading $\tau^{-1}$ in one-line notation right to left: in the indexing of $D(w)$, fill down column 4, then down column 2, then down column 1. The elements of each set $\mathcal{T}_w(r)$ are column-strict fillings in the corresponding diagrams $O(w,r)$:
		\begin{center}
			\includegraphics[scale=1]{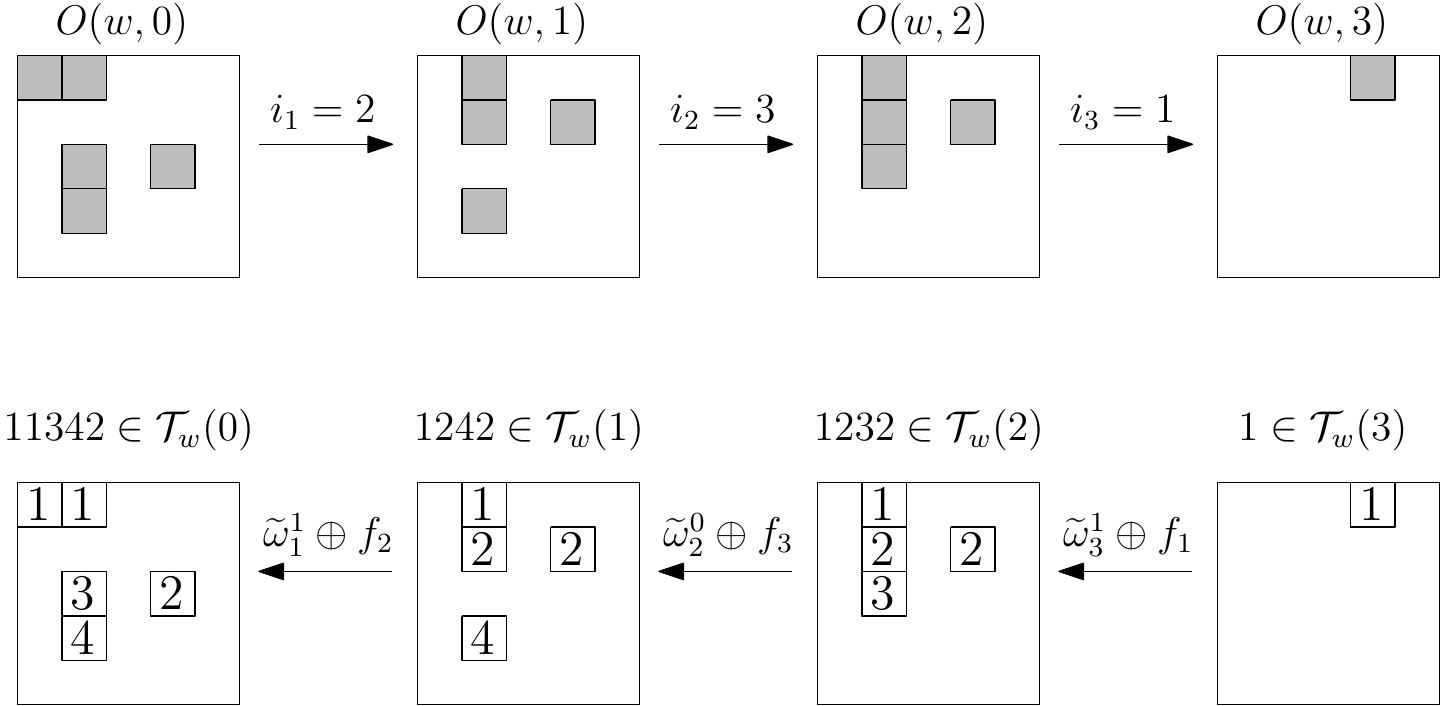}
		\end{center}
	\end{example}
	
	\begin{lemma}
		Let $w$ be a permutation with orthodontic sequence $(\bm{i},\bm{m})$, $\bm{i}=(i_1,\ldots,i_l)$. For each $0\leq r \leq l$, $O(w,r)$ has the northwest property.
	\end{lemma}
	
	\begin{definition}
		A filling $T$ of a diagram $D$ is called \emph{row-flagged} if $T(p,q)\leq p$ for each box $(p,q)\in D$.
	\end{definition}
	
	\begin{lemma}
		\label{lem:rowflagged}
		For each $0\leq r\leq l$, the elements of $\mathcal{T}_w(r)$ are row-flagged fillings of $O(w,r)$.
	\end{lemma}
	\begin{proof}
		Clearly, the singleton $\mathcal{T}_w(l)$ is a row-flagged filling of $O(w,l)$. Assume that for some $l\geq s>0$, the result holds with $r=s$. We show that the result also holds with $r=s-1$. Let $T\in\mathcal{T}_w(s)$. We must show that for each $u$, if $f_{i_s}^u(T)$ is defined, then $\widetilde{\omega}_{i_{s-1}}^{m_{s-1}}\oplus f_{i_s}^u(T)$ is a row-flagged filling of $O(w,s-1)$. By the orthodontia construction, $O(w,s)$ is obtained from $O(w,s-1)$ by removing the $m_{s-1}$ columns with no missing tooth, and then switching rows $i_s+1$ and $i_s$. 
		
		Since $T$ is a row-flagged filling of $O(w,s)$, each box in $O(w,s)$ containing an entry of $T$ equal to $i_s$ lies in a row with index at least $i_s$. Any box in $O(w,s)$ containing an entry of $T$ equal to $i_s$ and lying in row $i_s$ of $O(w,s)$ will have row index $i_s+1$ in $O(w,s-1)$. Any box in $O(w,s)$ containing an entry in $T$ equal to $i_s$ and lying in a row $d>i_s$ of $O(w,s)$ will still have row index $d$ in $O(w,s-1)$. Then if $f_{i_s}^u(T)$ is defined, $\widetilde{\omega}_{i_{s-1}}^{m_{s-1}}\oplus f_{i_s}^u(T)$ will be a row-flagged filling of $O(w,s-1)$.
	\end{proof}
	
\section{Zero-one Schubert polynomials}
	\label{sec:suff}
	
	This section is devoted to giving a sufficient condition on the orthodontic sequence $(\bm{i},\bm{m})$ of $w$ for the Schubert polynomial $\mathfrak{S}_w$ to be zero-one. We give such a condition in Theorem~\ref{thm:multfree}. We will see in Theorem~\ref{thm:011} that this condition turns out to also be a necessary condition for $\mathfrak{S}_w$ to be zero-one. 
	
	
	We start with a less ambitious result:	 
	\begin{proposition}
		\label{prop:norepeats}
		Let $w\in S_n$ and $(\bm{i},\bm{m})$ be the orthodontic sequence of $w$. If $\bm{i}=(i_1,\ldots,i_l)$ has distinct entries, then $\mathfrak{S}_w$ is zero-one.
	\end{proposition}
	\begin{proof}
		Let $T,T'\in\mathcal{T}_w$ with $wt(T)=wt(T')$. By Definition~\ref{def:magyartableauxset}, we can find $p_1,\ldots,p_l$ so that 
		\[T=\widetilde{\omega}_1^{k_1}\oplus\cdots\oplus \widetilde{\omega}_n^{k_n}\oplus f_{i_1}^{p_1}(\widetilde{\omega}_{i_1}^{m_1}\oplus\cdots\oplus f_{i_l}^{p_l}(\widetilde{\omega}_{i_l}^{m_l})\cdots). \]
		Then if $e_1,\ldots,e_n$ denote the standard basis vectors of $\mathbb{R}^n$,
		\[wt(T) = \sum_{j=1}^{n}wt(\widetilde{\omega}_{j}^{k_j}) + \sum_{j=1}^{l}wt(\widetilde{\omega}_{i_j}^{m_j})+\sum_{j=1}^{l}p_j(e_{i_j+1}-e_{i_j}). \]
		Similarly, we can find $q_1,\ldots,q_l$ so that 
		\[T'=\widetilde{\omega}_1^{k_1}\oplus\cdots\oplus \widetilde{\omega}_n^{k_n}\oplus f_{i_1}^{q_1}(\widetilde{\omega}_{i_1}^{m_1}\oplus\cdots\oplus f_{i_l}^{q_l}(\widetilde{\omega}_{i_l}^{m_l})\cdots), \]
		which implies
		\[wt(T') = \sum_{j=1}^{n}wt(\widetilde{\omega}_{j}^{k_j}) + \sum_{j=1}^{l}wt(\widetilde{\omega}_{i_j}^{m_j})+\sum_{j=1}^{l}q_j(e_{i_j+1}-e_{i_j}). \]
		As $wt(T)=wt(T')$,
		\begin{align}
			\label{eqn:vectorexpansion}
			0=wt(T)-wt(T')=(p_1-q_1)(e_{i_1+1}-e_{i_1})+\cdots+(p_l-q_l)(e_{i_l+1}-e_{i_l})\tag{$*$}.
		\end{align}
		Since the vectors $\{e_{i_j+1}-e_{i_j}\}_{j=1}^{l}$ are independent and $\bm{i}$ has distinct entries, $p_j=q_j$ for all $j$. Thus $T=T'$. This shows that all elements of $\mathcal{T}_w$ have distinct weights, so $\mathfrak{S}_w$ is zero-one.
	\end{proof}
	
	We now strengthen Proposition~\ref{prop:norepeats} to allow $\bm{i}$ to not have distinct entries. To do this, we will need a technical definition related to the orthodontic sequence. Recall the construction of the orthodontic sequence $(\bm{i},\bm{m})$ of a permutation $w\in S_n$ (Definition~\ref{def:imsequence}) and the intermediate diagrams $O(w,r)$ (Definition~\ref{def:partialdiagram}). 
	Let $\bm{i}=(i_1,\ldots,i_l)$, and define $O(w,r)_-$ to be the diagram $O(w,r)$ with all columns of the form $[i_{r}]$ replaced by empty columns.
		
	\begin{definition}
		Define the \emph{orthodontic impact function} $\mathcal{I}_w:[l]\to 2^{[n]}$ by  
		\[\mathcal{I}_w(j)=\{c\in [n] \mid (i_j+1,c)\in O(w,j-1)_-\}. \]
	\end{definition}
	\noindent That is, $\mathcal{I}_w(j)$ is the set of indices of columns of $O(w,j-1)_-$ that are changed when rows $i_j$ and $i_j+1$ are swapped to form $O(w,j)$. 
	
	\begin{definition}
		Let $w\in S_n$ have orthodontic sequence $(\bm{i},\bm{m})$, $\bm{i}=(i_1,\ldots,i_l)$. We say $w$ is \emph{multiplicity-free} if for any $r,s\in [l]$ with $r\neq s$ and $i_r=i_s$, we have $\mathcal{I}_w(r)=\mathcal{I}_w(s)=\{c\}$ for some $c\in [n]$.
	\end{definition}

	\begin{example}
		If $w=457812693$, then 
		\begin{center}
			\begin{tikzpicture}[scale=.5]
				\draw (0,0)--(9,0)--(9,9)--(0,9)--(0,0);
				\filldraw[draw=black,fill=lightgray] (0,8)--(1,8)--(1,9)--(0,9)--(0,8);
				\filldraw[draw=black,fill=lightgray] (1,8)--(2,8)--(2,9)--(1,9)--(1,8);
				\filldraw[draw=black,fill=lightgray] (2,8)--(3,8)--(3,9)--(2,9)--(2,8);
				\filldraw[draw=black,fill=lightgray] (0,7)--(1,7)--(1,8)--(0,8)--(0,7);
				\filldraw[draw=black,fill=lightgray] (1,7)--(2,7)--(2,8)--(1,8)--(1,7);
				\filldraw[draw=black,fill=lightgray] (2,7)--(3,7)--(3,8)--(2,8)--(2,7);
				\filldraw[draw=black,fill=lightgray] (0,6)--(1,6)--(1,7)--(0,7)--(0,6);
				\filldraw[draw=black,fill=lightgray] (1,6)--(2,6)--(2,7)--(1,7)--(1,6);
				\filldraw[draw=black,fill=lightgray] (2,6)--(3,6)--(3,7)--(2,7)--(2,6);
				\filldraw[draw=black,fill=lightgray] (0,5)--(1,5)--(1,6)--(0,6)--(0,5);
				\filldraw[draw=black,fill=lightgray] (1,5)--(2,5)--(2,6)--(1,6)--(1,5);
				\filldraw[draw=black,fill=lightgray] (2,5)--(3,5)--(3,6)--(2,6)--(2,5);
				\filldraw[draw=black,fill=lightgray] (5,6)--(6,6)--(6,7)--(5,7)--(5,6);
				\filldraw[draw=black,fill=lightgray] (5,5)--(6,5)--(6,6)--(5,6)--(5,5);
				\filldraw[draw=black,fill=lightgray] (2,2)--(3,2)--(3,3)--(2,3)--(2,2);
				\filldraw[draw=black,fill=lightgray] (2,1)--(3,1)--(3,2)--(2,2)--(2,1);
				\node at (-1.5,4.5) {$D(w)=$};
				\node at (13.5,4.5) {and $\bm{i}=(6,5,7,6,2,1,3,2).$};
			\end{tikzpicture}
		\end{center}
		The only entries of $\bm{i}$ occurring multiple times are $i_1=i_4=6$ and $i_5=i_8=2$. Their respective impacts are $\mathcal{I}_w(1)=\mathcal{I}_w(4)=\{3\}$ and $\mathcal{I}_w(5)=\mathcal{I}_w(8)=\{6\}$, so $w$ is multiplicity-free.
	\end{example}

	The proof of the generalization of Proposition~\ref{prop:norepeats} will require the following technical lemma. Before proceeding, recall Lemma~\ref{lem:fillings} and Lemma~\ref{lem:rowflagged}: for every $0\leq j\leq l$, elements of $\mathcal{T}_w(j)$ can be viewed as row-flagged, column-strict fillings of $O(w,j)$ (via the column filling order of $O(w,j)$ specified prior to Lemma~\ref{lem:fillings}). Applying $\widetilde{\omega}_{i_{j-1}}^{m_{j-1}}\oplus f_{i_j}$ to an element of $\mathcal{T}_w(j)$ gives an element of $\mathcal{T}_w(j-1)$, a filling of $O(w,j-1)$. Thus, when we speak below of the application of $f_{i_j}$ to an element $T\in\mathcal{T}_w(j)$ ``changing an $i_j$ to an $i_j+1$ in column $c$'', we specifically mean that when we view $T$ as a filling of $O(w,j)$ and $\widetilde{\omega}_{i_{j-1}}^{m_{j-1}}\oplus f_{i_j}(T)$ as a filling of $O(w,j-1)$, $T$ and $\widetilde{\omega}_{i_{j-1}}^{m_{j-1}}\oplus f_{i_j}(T)$ differ (in the stated way) in their entries in column $c$.
	\begin{lemma}
		\label{lem:rootoperatorproperty}
		Let $w$ be a multiplicity-free permutation with orthodontic sequence $(\bm{i},\bm{m})$, $\bm{i}=(i_1,\ldots,i_l)$. Suppose $i_r=i_s$ with $r<s$ and $\mathcal{I}_w(r)=\mathcal{I}_w(s)=\{c\}$.
		Then for each $j$ with $r\leq j \leq s$, $\mathcal{I}_w(j)=\{c\}$ and the application of $f_{i_j}$ to an element of 
		$\mathcal{T}_w(j)$
		is either undefined or changes an $i_j$ to an $i_{j}+1$ in column $c$.
	\end{lemma}
	\begin{proof}
		We handle first the case that $j=r$. In the diagram $O(w,r-1)$, column $c$ is the leftmost column containing a missing tooth, and $i_r$ is the smallest missing tooth in column $c$. Reading column $c$ of $O(w,r-1)$ top-to-bottom, one sees a (possibly empty) sequence of boxes in $O(w,r-1)$, followed by a sequence of boxes not in $O(w,r-1)$. The sequence of boxes not in $O(w,r-1)$ has length at least two since $i_r$ occurs at least twice in $\bm{i}$, and terminates with the box $(i_r+1,c)\in O(w,r-1)$. Note that since $(i_r-1,c),(i_r,c)\notin O(w,r-1)$,  the northwest property of $O(w,r-1)$ implies that there can be no box $(i_r-1,c')$ or $(i_r,c')$ in $O(w,r-1)$ with $c'>c$. 
		Note also that since $\mathcal{I}_w(r)=\{c\}$, we have $(i_r+1,c')\notin O(w,r-1)$ for each $c'>c$. 
		Lastly, observe that for any $c'>c$ and $d>i_r+1$, there can be no box $(d,c')\in O(w,r-1)$. Otherwise there would be some $t\in [l]$ with $i_t=i_r$ and $t\neq r$ such that $c'\in \mathcal{I}_w(t)$, violating that $w$ is multiplicity-free. 
		
		As a consequence of the previous observations, the largest row index that a column $c'>c$ of $O(w,r-1)$ can contain a box in is $i_r-2$. In particular, Lemma~\ref{lem:rowflagged} implies that the application of $f_{i_r}$ to an element of $\mathcal{T}_w(r)$ either is undefined or changes an $i_r$ to an $i_r+1$ in column $c$. This concludes the case that $j=r$. 
		
		When $j=s$, an entirely analogous argument works. The only significant difference in the observations is that when column $c$ of $O(w,s-1)$ is read top-to-bottom, the (possibly empty) initial sequence of boxes in $O(w,s-1)$ is followed by a sequence of boxes not in $O(w,s-1)$ with length at least 1, ending with the box $(i_s+1,c)$. Consequently, the largest row index that a column $c'>c$ of $O(w,s-1)$ can contain a box in is $i_s-1$. In particular, Lemma~\ref{lem:rowflagged} implies that the application of $f_{i_s}$ to an element of $\mathcal{T}_w(s)$ either is undefined or changes an $i_s$ to an $i_s+1$ in column $c$. This concludes the case that $j=s$. 
		
		Now, let $r<j<s$. Since $\mathcal{I}_w(r)=\mathcal{I}_w(s)=\{c\}$, we have $c\in \mathcal{I}_w(j)$. 
		If $i_j$ occurs multiple times in $\bm{i}$, then multiplicity-freeness of $w$ implies $\mathcal{I}_w(j)=\{c\}$. In this case, we can find $j'\neq j$ with $i_j=i_{j'}$ and apply the previous argument (with $r$ and $s$ replaced by $j$ and $j'$) to conclude that the application of $f_{i_j}$ to an element of $\mathcal{T}_w(j)$ is either undefined or changes an $i_j$ to an $i_j+1$ in column $c$.
		
		Thus, we assume $i_j$ occurs only once in $\bm{i}$. Recall that it was shown above that $O(w,r-1)$ has no boxes $(d,c')$ with $d>i_r$ and $c'>c$. Read top-to-bottom, let column $c$ of $O(w,r-1)$ have a (possibly empty) initial sequence of boxes ending with a missing box in row $u$, so clearly $u\leq i_r-1$. Since the first missing tooth in column $c$ of $O(w,r-1)$ is in row $i_r$, none of the boxes $(u,c),(u+1,c),\ldots,(i_r,c)$ are in $O(w,r-1)$, but $(i_r+1,c)\in O(w,r-1)$. Then, the northwest property implies that there is no box in $O(w,r-1)$ in any column $c'>c$ in any of rows $u,u+1,\ldots,i_r$. In particular, the largest row index such that a column $c'>c$ of $O(w,r-1)$ can contain a box in is $u-1$. 
		
		As $r<j<s$ and $\mathcal{I}_w(r)=\mathcal{I}_w(s)=\{c\}$, we have that $c\in\mathcal{I}_w(j)$. 
		Also since $r<j<s$, the leftmost nonempty column in $O(w,j-1)$ is column $c$, and $i_j\geq u$. Then in $O(w,j-1)$, the maximum row index a box in a column $c'>c$ can have is $u-1$. In particular, $\mathcal{I}_w(j)=\{c\}$, and Lemma~\ref{lem:rowflagged} implies that the application of $f_{i_j}$ to an element of $\mathcal{T}_w(j)$ is either undefined or changes an $i_j$ to an $i_j+1$ in column $c$.
	\end{proof}

	\begin{theorem}
		\label{thm:multfree}
		If $w$ is multiplicity-free, then $\mathfrak{S}_w$ is zero-one.
	\end{theorem}

	\begin{proof}
		Assume $wt(T)=wt(T')$ for some $T,T'\in \mathcal{T}_w$. If we can show that $T=T'$, then we can conclude that all elements of $\mathcal{T}_w$ have distinct weights, so $\mathfrak{S}_w$ is zero-one. To begin, write 
		\[T=\widetilde{\omega}_1^{k_1}\oplus\cdots\oplus \widetilde{\omega}_n^{k_n}\oplus f_{i_1}^{p_1}(\widetilde{\omega}_{i_1}^{m_1}\oplus\cdots\oplus f_{i_l}^{p_l}(\widetilde{\omega}_{i_l}^{m_l})\cdots) \]
		and
		\[T'=\widetilde{\omega}_1^{k_1}\oplus\cdots\oplus \widetilde{\omega}_n^{k_n}\oplus f_{i_1}^{q_1}(\widetilde{\omega}_{i_1}^{m_1}\oplus\cdots\oplus f_{i_l}^{q_l}(\widetilde{\omega}_{i_l}^{m_l})\cdots), \]
		for some $p_1,\ldots,p_l,q_1,\ldots,q_l$. The basic idea of the proof is to show that as $T$ and $T'$ are constructed step-by-step from $\widetilde{\omega}_{i_l}^{m_l}$, the resulting intermediate tableaux are intermittently equal. At termination of the construction, this will imply that $T=T'$. 
		
		By the expansion (\ref{eqn:vectorexpansion}) of $wt(T)-wt(T')$ used in the proof of Proposition~\ref{prop:norepeats}, we observe that $p_u=q_u$ for all $u$ such that $i_u$ occurs only once in $\bm{i}$. Let $s$ be the largest index such that $p_s\neq q_s$. Suppose $\mathcal{I}_w(s)=\{c\}$. 
		Let $r_1$ be the smallest index such that $i_{r_1}$ occurs multiple times in $\bm{i}$ and $\mathcal{I}_w(r_1)=\{c\}$.
		We know $r_1<s$, because (\ref{eqn:vectorexpansion}) implies that $p_{s'}\neq q_{s'}$ for another $s'<s$ with $i_{s'}=i_s$, and by multiplicity-freeness $\mathcal{I}_w(s')=\{c\}$.
		We wish to find an interval $[r,s]\subseteq [r_1,s]$ such that $r<s$ and the following two conditions hold:
		\begin{itemize}
			\item[(i)] \label{itm:i} If $v\geq r$ and $i_v$ occurs multiple times in $\bm{i}$, then any other $v'$ with $i_v=i_{v'}$ will satisfy $v'\geq r$.
			\item[(ii)] \label{itm:ii} For every $j$ with $r<j<s$ and $i_j$ occurring only once in $\bm{i}$, there are $t$ and $u$ with $r\leq t<j<u\leq s$ such that $i_t=i_u$.
		\end{itemize}
			
		We first show that (i) holds for $[r_1,s]$. Note that if $i_v$ occurs multiple times in $\bm{i}$ and $r_1\leq v\leq s$, then it must be that $\mathcal{I}_w(v)=\{c\}$		
		by the fact that the orthodontia construction records all missing teeth needed to eliminate one column before moving on to the next column.
		If $i_{v'}=i_{v}$, then $\mathcal{I}_w(v')=\{c\}$ also, by multiplicity-freeness of $w$. The choice of $r_1$ implies $r_1\leq v'$. 
		If $i_v$ occurs multiple times in $\bm{i}$ with $s<v$ and $\mathcal{I}_w(v)=\{c\}$, then the choice of $r_1$ again implies that $r_1\leq v'$ for any $i_{v'}=i_v$. 
		If $i_v$ occurs multiple times in $\bm{i}$ with $s<v$ and $\mathcal{I}_w(v)\neq \{c\}$, then the orthodontia construction implies that 
		any $v'$ with $i_v=i_{v'}$ must satisfy $s<v'$. 
		In particular, $r_1<v'$ as needed. Thus, (i) holds for $[r_1,s]$. If $[r_1,s]$ also satisfies (ii), then we are done.
		
		Otherwise, assume $[r_1,s]$ does not satisfy (ii). Then there is some $j$ with $r_1<j<s$ such that $i_j$ occurs only once in $\bm{i}$ and
		there are no $t$ and $u$ with $r_1\leq t<j<u\leq s$ and $i_t=i_u$. Consequently for every pair $i_u=i_t$ with $r_1\leq t<u\leq s$, it must be that either $t<u<j$ or $j<t<u$. Let $r_2$ be the smallest index such that $j<r_2$ and $i_{r_2}$ occurs multiple times in $\bm{i}$. By the choice of $j$, it is clear that the interval $[r_2,s]$ still satisfies (i). If $[r_2,s]$ also satisfies (ii), then we are done. 

		Otherwise, $[r_2,s]$ satisfies (i) but not (ii), and we can argue exactly as in the case of $[r_1,s]$ to find an $r_3$ such that $r_2<r_3<s$ and $[r_3,s]$ satisfies (i). 
		Continue working in this fashion. We show that this process terminates with an interval $[r,s]$ satisfying $r<s$, (i), and (ii). 

		As mentioned above, there exists $s'<s$ such that $i_{s'}=i_s$. Let $s'$ be the maximal index less than $s$ with this property. 
		Since all of the intervals $[r_*,s]$ will satisfy (i), it follows that $r_1<r_2<\cdots\leq s'$. At worst, the  process will terminate after finitely many steps with the interval $[s',s]$. The interval $[s',s]$ will then satisfy (i) since the process reached it, and will trivially satisfy (ii) since $i_s=i_{s'}$.		
		
		Hence, we can assume that we have found an interval $[r,s]$ satisfying $r<s$, (i), and (ii). Consider the tableaux
		\begin{align*}
		&T_r=\widetilde{\omega}_{i_{r-1}}^{m_{r-1}}\oplus f_{i_r}^{p_r}(\widetilde{\omega}_{i_r}^{m_r}\oplus\cdots\oplus f_{i_l}^{p_l}(\widetilde{\omega}_{i_l}^{m_l})\cdots),\quad &T_s=\widetilde{\omega}_{i_{s}}^{m_{s}}\oplus f_{i_{s+1}}^{p_{s+1}}(\widetilde{\omega}_{i_{s+1}}^{m_{s+1}}\oplus\cdots\oplus f_{i_l}^{p_l}(\widetilde{\omega}_{i_l}^{m_l})\cdots),\\
		&T'_r=\widetilde{\omega}_{i_{r-1}}^{m_{r-1}}\oplus f_{i_r}^{q_r}(\widetilde{\omega}_{i_r}^{m_r}\oplus\cdots\oplus f_{i_l}^{q_l}(\widetilde{\omega}_{i_l}^{m_l})\cdots),\quad &T'_s=\widetilde{\omega}_{i_{s}}^{m_{s}}\oplus f_{i_{s+1}}^{q_{s+1}}(\widetilde{\omega}_{i_{s+1}}^{m_{s+1}}\oplus\cdots\oplus f_{i_l}^{q_l}(\widetilde{\omega}_{i_l}^{m_l})\cdots).
		\end{align*}
		
		By definition, $T_r,T_r'\in\mathcal{T}_w(r-1)$, so we can view $T_r$ and $T_r'$ as fillings of $O(w,r-1)$. Similarly, $T_s,T_s'\in\mathcal{T}_w(s)$, so we can view $T_s$ and $T_s'$ as fillings of $O(w,s)$. Since we chose $s$ to be the largest index such that $p_s\neq q_s$, it follows that $T_s=T'_s$. By property (i) of $[r,s]$, $i_u\neq i_v$ for any $u<r\leq v$. Hence, it must be that $wt(T_r)=wt(T'_r)$. Finally, property (ii) of $[r,s]$ allows us to apply Lemma~\ref{lem:rootoperatorproperty} and conclude that for any $a_r,a_{r+1},\ldots,a_{s}\geq 0$, 
		when $\widetilde{\omega}_{i_{r-1}}^{m_{r-1}}\oplus f_{i_r}^{a_r}(\widetilde{\omega}_{m_r}^{i_r}\oplus\cdots\oplus \widetilde{\omega}_{i_{s-1}}^{m_{s-1}}f_{i_s}^{a_s}(\mbox{---})\cdots)$ 
		is applied to an element of $\mathcal{T}_w(s)$, only the entries in column $c$ are affected by the root operators $f_{i_r}^{a_r},\ldots,f_{i_s}^{a_s}$. Since
		\[
			T_r= \widetilde{\omega}_{i_{r-1}}^{m_{r-1}}\oplus f_{i_r}^{p_r}(\widetilde{\omega}_{m_r}^{i_r}\oplus\cdots\oplus \widetilde{\omega}_{i_{s-1}}^{m_{s-1}}f_{i_s}^{p_s}(T_s)\cdots)
			\quad\mbox{and}\quad
			T_r'= \widetilde{\omega}_{i_{r-1}}^{m_{r-1}}\oplus f_{i_r}^{q_r}(\widetilde{\omega}_{m_r}^{i_r}\oplus\cdots\oplus \widetilde{\omega}_{i_{s-1}}^{m_{s-1}}f_{i_s}^{q_s}(T_s')\cdots),
		\]	
		$T_r$ and $T_r'$ must coincide outside of column $c$. Since we already deduced that $wt(T_r)=wt(T'_r)$, it follows that column $c$ of $T_r$ and $T_r'$ have the same weight. By column-strictness of $T_r$ and $T_r'$, column $c$ of $T_r$ and $T_r'$ must coincide, so $T_r=T_r'$.
		
		To complete the proof, let $\hat{s}$ be the largest index $\hat{s}<r$ such that $p_{\hat{s}}\neq q_{\hat{s}}$. If no such index exists, then $T=T'$. Otherwise, set $\hat{r}_1$ to be the smallest index such that $i_{\hat{r}_1}$ occurs multiple times in $\bm{i}$ and $\mathcal{I}_w(\hat{r}_1)=\mathcal{I}_w(\hat{s})$. 
		We have $\hat{r}_1<\hat{s}$ because some other $\hat{s}'$ distinct from~$\hat{s}$ such that $p_{\hat{s}'}\neq q_{\hat{s}'}$ and $i_{\hat{s}'}=i_{\hat{s}}$ must exist as before,
		and $\hat{s}'$ is also less than~$r$ by property (i) of $[r,s]$.
		Use the previous algorithm to find an interval $[\hat{r},\hat{s}]\subseteq [\hat{r}_1,\hat{s}]$ satisfying $\hat{r}<\hat{s}$, (i), and (ii). Construct $T_{\hat{r}},T_{\hat{r}}',T_{\hat{s}},T'_{\hat{s}}$, and argue exactly as in the case of $[r,s]$ that $T_{\hat{r}}=T_{\hat{r}}'$.
		
		Continuing in this manner for a finite number of steps will show that $T=T'$. 
	\end{proof}
	
	As we will show  in Theorem~\ref{thm:011},
	it is not only sufficient but also necessary that $w$ be multiplicity-free for the Schubert polynomial $\mathfrak{S}_w$ to be zero-one.

\section{Pattern avoidance conditions for multiplicity-freeness}
	\label{sec:mult-patt}
	This section is devoted to showing that $w$ being multiplicity-free is equivalent to a certain pattern avoidance condition. We then prove our full characterization of zero-one Schubert polynomials. 
	
	
	We start with several definitions. 
	\begin{definition}
		\label{def:confA}
		We say a Rothe diagram $D=D(w)$ contains an instance of configuration $\mathrm{A}$ if there are $r_1,c_1,r_2,c_2,r_3$ such that $1\leq r_3<r_1<r_2$, $1<c_1<c_2$, $(r_1,c_1),(r_2,c_2)\in D$, $(r_1,c_2)\notin D$, and $w_{r_3}<c_1$. 
	\end{definition}
	\begin{definition}
		\label{def:confB}
		We say a Rothe diagram $D=D(w)$ contains an instance of configuration $\mathrm{B}$ if there are $r_1,c_1,r_2,c_2,r_3,r_4$ such that $1\leq r_4\neq r_3<r_1<r_2$, $2<c_1<c_2$, $(r_1,c_1),(r_1,c_2),(r_2,c_2)\in D$, $w_{r_3}<c_1$, and $w_{r_4}<c_2$.
	\end{definition}
	\begin{definition}
		\label{def:confBprime}
		We say a Rothe diagram $D=D(w)$ contains an instance of configuration $\mathrm{B}'$ if there are $r_1,c_1,r_2,c_2,r_3,r_4$ such that $1\leq r_4<r_3<r_1<r_2$, $2<c_1<c_2$, $(r_1,c_1),(r_1,c_2),(r_2,c_1)\in D$, $w_{r_3}<c_1$, and $w_{r_4}<c_1$.
	\end{definition}
	Given a Rothe diagram $D(w)$, we will call a tuple $(r_1,c_1,r_2,c_2,r_3)$ meeting the conditions of Definition~\ref{def:confA} an instance of configuration $\mathrm{A}$ in $D(w)$. Similarly, we will call a tuple $(r_1,c_1,r_2,c_2,r_3,r_4)$ meeting the conditions of Definition~\ref{def:confB} (resp. \ref{def:confBprime}) an instance of configuration $\mathrm{B}$ (resp. $\mathrm{B}'$) in $D(w)$.

	\begin{figure}[ht]
		\centering
		\begin{tikzpicture}[scale=.75]
			\draw (0,0)--(5,0)--(5,5)--(0,5)--(0,0);
			\draw (5,4.5) -- (0.5,4.5) node {$\bullet$} -- (0.5,0);
			\draw (5,3.5) -- (2.5,3.5) node {$\bullet$} -- (2.5,0);
			\draw (5,2.5) -- (1.5,2.5) node {$\bullet$} -- (1.5,0);
			\draw (5,1.5) -- (4.5,1.5) node {$\bullet$} -- (4.5,0);
			\draw (5,0.5) -- (3.5,0.5) node {$\bullet$} -- (3.5,0);
			
			\filldraw[draw=black,fill=lightgray] (1,3)--(2,3)--(2,4)--(1,4)--(1,3);
			\filldraw[draw=black,fill=lightgray] (3,1)--(4,1)--(4,2)--(3,2)--(3,1);
			
			\node at (1.5,5.25) {$c_1$};
			\node at (3.5,5.25) {$c_2$};
			\node at (-.3,4.5) {$r_3$};
			\node at (-.3,3.5) {$r_1$};
			\node at (-.3,1.5) {$r_2$};
			\node at (2.5,-.6) {\large $\mathrm{A}$};
		\end{tikzpicture}
		\quad
		\begin{tikzpicture}[scale=.625]
			\draw (0,0)--(6,0)--(6,6)--(0,6)--(0,0);
			\draw (6,5.5) -- (2.5,5.5) node {$\bullet$} -- (2.5,0);
			\draw (6,4.5) -- (0.5,4.5) node {$\bullet$} -- (0.5,0);
			\draw (6,3.5) -- (4.5,3.5) node {$\bullet$} -- (4.5,0);
			\draw (6,2.5) -- (1.5,2.5) node {$\bullet$} -- (1.5,0);
			\draw (6,1.5) -- (5.5,1.5) node {$\bullet$} -- (5.5,0);
			\draw (6,0.5) -- (3.5,0.5) node {$\bullet$} -- (3.5,0);
			
			\filldraw[draw=black,fill=lightgray] (0,5)--(1,5)--(1,6)--(0,6)--(0,5);
			\filldraw[draw=black,fill=lightgray] (1,5)--(2,5)--(2,6)--(1,6)--(1,5);
			\filldraw[draw=black,fill=lightgray] (1,3)--(1,4)--(2,4)--(2,3)--(1,3);
			\filldraw[draw=black,fill=lightgray] (3,3)--(4,3)--(4,4)--(3,4)--(3,3);
			\filldraw[draw=black,fill=lightgray] (3,1)--(4,1)--(4,2)--(3,2)--(3,1);
			
			\node at (1.5,6.25) {$c_1$};
			\node at (3.5,6.25) {$c_2$};
			\node at (-.3,5.5) {$r_4$};
			\node at (-.3,4.5) {$r_3$};
			\node at (-.3,3.5) {$r_1$};
			\node at (-.3,1.5) {$r_2$};
			\node at (3,-.65) {\large $\mathrm{B}$};
		\end{tikzpicture}
		\quad
		\begin{tikzpicture}[scale=.75]
			\draw (0,0)--(5,0)--(5,5)--(0,5)--(0,0);
			\draw (5,4.5) -- (0.5,4.5) node {$\bullet$} -- (0.5,0);
			\draw (5,3.5) -- (1.5,3.5) node {$\bullet$} -- (1.5,0);
			\draw (5,2.5) -- (4.5,2.5) node {$\bullet$} -- (4.5,0);
			\draw (5,1.5) -- (3.5,1.5) node {$\bullet$} -- (3.5,0);
			\draw (5,0.5) -- (2.5,0.5) node {$\bullet$} -- (2.5,0);
			
			\filldraw[draw=black,fill=lightgray] (2,1)--(3,1)--(3,2)--(2,2)--(2,1);
			\filldraw[draw=black,fill=lightgray] (2,2)--(3,2)--(3,3)--(2,3)--(2,2);
			\filldraw[draw=black,fill=lightgray] (3,2)--(4,2)--(4,3)--(3,3)--(3,2);
			
			\node at (2.5,5.25) {$c_1$};
			\node at (3.5,5.25) {$c_2$};
			\node at (-.3,4.5) {$r_4$};
			\node at (-.3,3.5) {$r_3$};
			\node at (-.3,2.5) {$r_1$};
			\node at (-.3,1.5) {$r_2$};
			\node at (2.5,-.6) {\large $\mathrm{B}'$};
		\end{tikzpicture}
		\caption{Examples of instances of the configurations $\mathrm{A}$, $\mathrm{B}$, and $\mathrm{B}'$ in Rothe diagrams. Both the hooks removed from the $n\times n$ grid to form each Rothe diagram and the remaining boxes are shown.}
	\end{figure}
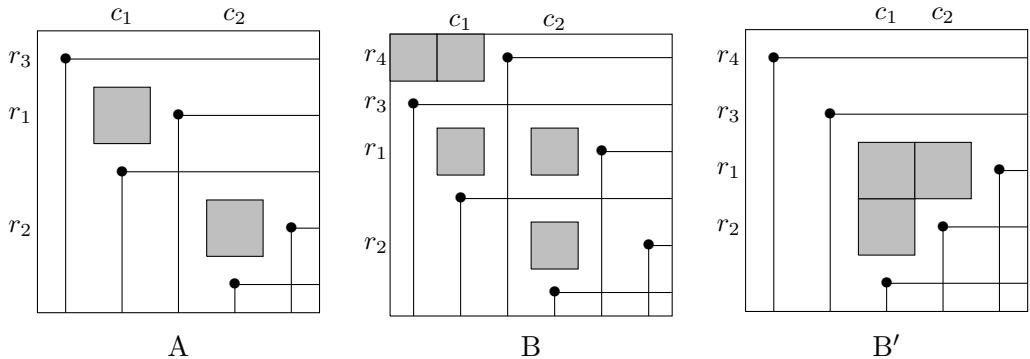
	
	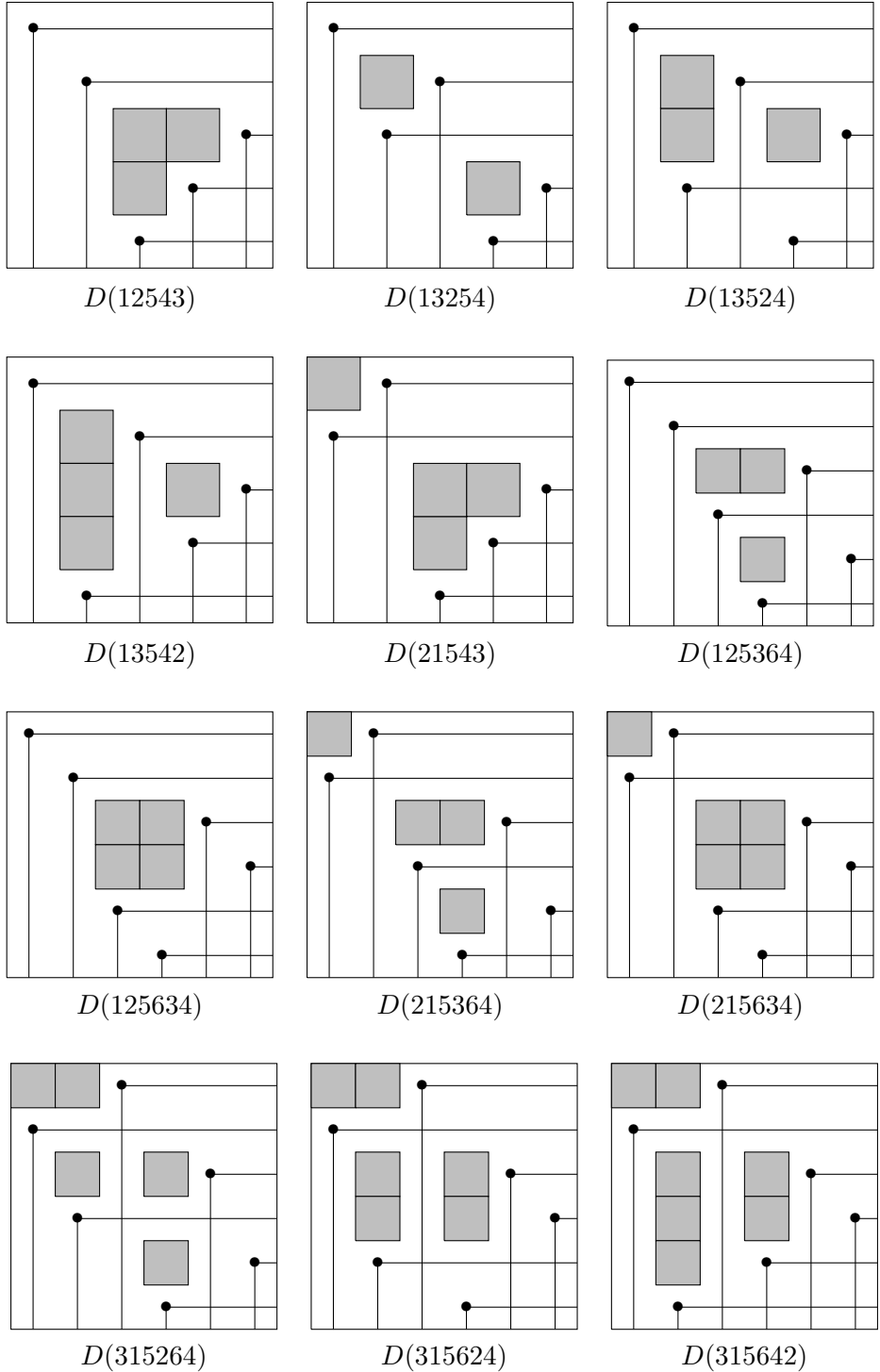
\begin{figure}[ht]
		\centering
		\begin{tikzpicture}[scale=.75]
		\draw (0,0)--(5,0)--(5,5)--(0,5)--(0,0);
		\draw (5,4.5) -- (0.5,4.5) node {$\bullet$} -- (0.5,0);
		\draw (5,3.5) -- (1.5,3.5) node {$\bullet$} -- (1.5,0);
		\draw (5,2.5) -- (4.5,2.5) node {$\bullet$} -- (4.5,0);
		\draw (5,1.5) -- (3.5,1.5) node {$\bullet$} -- (3.5,0);
		\draw (5,0.5) -- (2.5,0.5) node {$\bullet$} -- (2.5,0);
		
		\filldraw[draw=black,fill=lightgray] (2,1)--(3,1)--(3,2)--(2,2)--(2,1);
		\filldraw[draw=black,fill=lightgray] (2,2)--(3,2)--(3,3)--(2,3)--(2,2);
		\filldraw[draw=black,fill=lightgray] (3,2)--(4,2)--(4,3)--(3,3)--(3,2);
		
		\node at (2.5,-.6) {\large $D(12543)$};
		\end{tikzpicture}
		\quad
		\begin{tikzpicture}[scale=.75]
		\draw (0,0)--(5,0)--(5,5)--(0,5)--(0,0);
		\draw (5,4.5) -- (0.5,4.5) node {$\bullet$} -- (0.5,0);
		\draw (5,3.5) -- (2.5,3.5) node {$\bullet$} -- (2.5,0);
		\draw (5,2.5) -- (1.5,2.5) node {$\bullet$} -- (1.5,0);
		\draw (5,1.5) -- (4.5,1.5) node {$\bullet$} -- (4.5,0);
		\draw (5,0.5) -- (3.5,0.5) node {$\bullet$} -- (3.5,0);
		
		\filldraw[draw=black,fill=lightgray] (1,3)--(2,3)--(2,4)--(1,4)--(1,3);
		\filldraw[draw=black,fill=lightgray] (3,1)--(4,1)--(4,2)--(3,2)--(3,1);
		
		\node at (2.5,-.6) {\large $D(13254)$};
		\end{tikzpicture}
		\quad
		\begin{tikzpicture}[scale=.75]
		\draw (0,0)--(5,0)--(5,5)--(0,5)--(0,0);
		\draw (5,4.5) -- (0.5,4.5) node {$\bullet$} -- (0.5,0);
		\draw (5,3.5) -- (2.5,3.5) node {$\bullet$} -- (2.5,0);
		\draw (5,2.5) -- (4.5,2.5) node {$\bullet$} -- (4.5,0);
		\draw (5,1.5) -- (1.5,1.5) node {$\bullet$} -- (1.5,0);
		\draw (5,0.5) -- (3.5,0.5) node {$\bullet$} -- (3.5,0);
		
		\filldraw[draw=black,fill=lightgray] (1,3)--(2,3)--(2,4)--(1,4)--(1,3);
		\filldraw[draw=black,fill=lightgray] (1,2)--(2,2)--(2,3)--(1,3)--(1,2);
		\filldraw[draw=black,fill=lightgray] (3,2)--(4,2)--(4,3)--(3,3)--(3,2);
		
		\node at (2.5,-.6) {\large $D(13524)$};
		\end{tikzpicture}
		\vspace{3ex}
		
		\begin{tikzpicture}[scale=.75]
		\draw (0,0)--(5,0)--(5,5)--(0,5)--(0,0);
		\draw (5,4.5) -- (0.5,4.5) node {$\bullet$} -- (0.5,0);
		\draw (5,3.5) -- (2.5,3.5) node {$\bullet$} -- (2.5,0);
		\draw (5,2.5) -- (4.5,2.5) node {$\bullet$} -- (4.5,0);
		\draw (5,1.5) -- (3.5,1.5) node {$\bullet$} -- (3.5,0);
		\draw (5,0.5) -- (1.5,0.5) node {$\bullet$} -- (1.5,0);
		
		\filldraw[draw=black,fill=lightgray] (1,3)--(2,3)--(2,4)--(1,4)--(1,3);
		\filldraw[draw=black,fill=lightgray] (1,2)--(2,2)--(2,3)--(1,3)--(1,2);
		\filldraw[draw=black,fill=lightgray] (1,1)--(2,1)--(2,2)--(1,2)--(1,1);
		\filldraw[draw=black,fill=lightgray] (3,2)--(4,2)--(4,3)--(3,3)--(3,2);
		
		\node at (2.5,-.6) {\large $D(13542)$};
		\end{tikzpicture}
		\quad
		\begin{tikzpicture}[scale=.75]
		\draw (0,0)--(5,0)--(5,5)--(0,5)--(0,0);
		\draw (5,4.5) -- (1.5,4.5) node {$\bullet$} -- (1.5,0);
		\draw (5,3.5) -- (0.5,3.5) node {$\bullet$} -- (0.5,0);
		\draw (5,2.5) -- (4.5,2.5) node {$\bullet$} -- (4.5,0);
		\draw (5,1.5) -- (3.5,1.5) node {$\bullet$} -- (3.5,0);
		\draw (5,0.5) -- (2.5,0.5) node {$\bullet$} -- (2.5,0);
		
		\filldraw[draw=black,fill=lightgray] (0,4)--(1,4)--(1,5)--(0,5)--(0,4);
		\filldraw[draw=black,fill=lightgray] (2,1)--(3,1)--(3,2)--(2,2)--(2,1);
		\filldraw[draw=black,fill=lightgray] (2,2)--(3,2)--(3,3)--(2,3)--(2,2);
		\filldraw[draw=black,fill=lightgray] (3,2)--(4,2)--(4,3)--(3,3)--(3,2);
		
		\node at (2.5,-.6) {\large $D(21543)$};
		\end{tikzpicture}
		\quad
		\begin{tikzpicture}[scale=.625]
		\draw (0,0)--(6,0)--(6,6)--(0,6)--(0,0);
		\draw (6,5.5) -- (0.5,5.5) node {$\bullet$} -- (0.5,0);
		\draw (6,4.5) -- (1.5,4.5) node {$\bullet$} -- (1.5,0);
		\draw (6,3.5) -- (4.5,3.5) node {$\bullet$} -- (4.5,0);
		\draw (6,2.5) -- (2.5,2.5) node {$\bullet$} -- (2.5,0);
		\draw (6,1.5) -- (5.5,1.5) node {$\bullet$} -- (5.5,0);
		\draw (6,0.5) -- (3.5,0.5) node {$\bullet$} -- (3.5,0);
		
		\filldraw[draw=black,fill=lightgray] (2,3)--(3,3)--(3,4)--(2,4)--(2,3);
		\filldraw[draw=black,fill=lightgray] (3,1)--(4,1)--(4,2)--(3,2)--(3,1);
		\filldraw[draw=black,fill=lightgray] (3,3)--(4,3)--(4,4)--(3,4)--(3,3);
		
		\node at (3,-.65) {\large $D(125364)$};
		\end{tikzpicture}
		\vspace{3ex}
		
		\begin{tikzpicture}[scale=.625]
		\draw (0,0)--(6,0)--(6,6)--(0,6)--(0,0);
		\draw (6,5.5) -- (0.5,5.5) node {$\bullet$} -- (0.5,0);
		\draw (6,4.5) -- (1.5,4.5) node {$\bullet$} -- (1.5,0);
		\draw (6,3.5) -- (4.5,3.5) node {$\bullet$} -- (4.5,0);
		\draw (6,2.5) -- (5.5,2.5) node {$\bullet$} -- (5.5,0);
		\draw (6,1.5) -- (2.5,1.5) node {$\bullet$} -- (2.5,0);
		\draw (6,0.5) -- (3.5,0.5) node {$\bullet$} -- (3.5,0);
		
		\filldraw[draw=black,fill=lightgray] (2,3)--(3,3)--(3,4)--(2,4)--(2,3);
		\filldraw[draw=black,fill=lightgray] (3,3)--(4,3)--(4,4)--(3,4)--(3,3);
		\filldraw[draw=black,fill=lightgray] (2,2)--(3,2)--(3,3)--(2,3)--(2,2);
		\filldraw[draw=black,fill=lightgray] (3,2)--(4,2)--(4,3)--(3,3)--(3,2);
		\node at (3,-.65) {\large $D(125634)$};
		\end{tikzpicture}
		\quad
		\begin{tikzpicture}[scale=.625]
		\draw (0,0)--(6,0)--(6,6)--(0,6)--(0,0);
		\draw (6,5.5) -- (1.5,5.5) node {$\bullet$} -- (1.5,0);
		\draw (6,4.5) -- (0.5,4.5) node {$\bullet$} -- (0.5,0);
		\draw (6,3.5) -- (4.5,3.5) node {$\bullet$} -- (4.5,0);
		\draw (6,2.5) -- (2.5,2.5) node {$\bullet$} -- (2.5,0);
		\draw (6,1.5) -- (5.5,1.5) node {$\bullet$} -- (5.5,0);
		\draw (6,0.5) -- (3.5,0.5) node {$\bullet$} -- (3.5,0);
		
		\filldraw[draw=black,fill=lightgray] (0,5)--(1,5)--(1,6)--(0,6)--(0,5);
		\filldraw[draw=black,fill=lightgray] (2,3)--(3,3)--(3,4)--(2,4)--(2,3);
		\filldraw[draw=black,fill=lightgray] (3,1)--(4,1)--(4,2)--(3,2)--(3,1);
		\filldraw[draw=black,fill=lightgray] (3,3)--(4,3)--(4,4)--(3,4)--(3,3);
		
		\node at (3,-.65) {\large $D(215364)$};
		\end{tikzpicture}
		\quad
		\begin{tikzpicture}[scale=.625]
		\draw (0,0)--(6,0)--(6,6)--(0,6)--(0,0);
		\draw (6,5.5) -- (1.5,5.5) node {$\bullet$} -- (1.5,0);
		\draw (6,4.5) -- (0.5,4.5) node {$\bullet$} -- (0.5,0);
		\draw (6,3.5) -- (4.5,3.5) node {$\bullet$} -- (4.5,0);
		\draw (6,2.5) -- (5.5,2.5) node {$\bullet$} -- (5.5,0);
		\draw (6,1.5) -- (2.5,1.5) node {$\bullet$} -- (2.5,0);
		\draw (6,0.5) -- (3.5,0.5) node {$\bullet$} -- (3.5,0);
		
		\filldraw[draw=black,fill=lightgray] (0,5)--(1,5)--(1,6)--(0,6)--(0,5);
		\filldraw[draw=black,fill=lightgray] (2,3)--(3,3)--(3,4)--(2,4)--(2,3);
		\filldraw[draw=black,fill=lightgray] (3,3)--(4,3)--(4,4)--(3,4)--(3,3);
		\filldraw[draw=black,fill=lightgray] (2,2)--(3,2)--(3,3)--(2,3)--(2,2);
		\filldraw[draw=black,fill=lightgray] (3,2)--(4,2)--(4,3)--(3,3)--(3,2);
		\node at (3,-.65) {\large $D(215634)$};
		\end{tikzpicture}
		\vspace{3ex}
		
		\begin{tikzpicture}[scale=.625]
		\draw (0,0)--(6,0)--(6,6)--(0,6)--(0,0);
		\draw (6,5.5) -- (2.5,5.5) node {$\bullet$} -- (2.5,0);
		\draw (6,4.5) -- (0.5,4.5) node {$\bullet$} -- (0.5,0);
		\draw (6,3.5) -- (4.5,3.5) node {$\bullet$} -- (4.5,0);
		\draw (6,2.5) -- (1.5,2.5) node {$\bullet$} -- (1.5,0);
		\draw (6,1.5) -- (5.5,1.5) node {$\bullet$} -- (5.5,0);
		\draw (6,0.5) -- (3.5,0.5) node {$\bullet$} -- (3.5,0);
		
		\filldraw[draw=black,fill=lightgray] (0,5)--(1,5)--(1,6)--(0,6)--(0,5);
		\filldraw[draw=black,fill=lightgray] (1,5)--(2,5)--(2,6)--(1,6)--(1,5);
		\filldraw[draw=black,fill=lightgray] (1,3)--(1,4)--(2,4)--(2,3)--(1,3);
		\filldraw[draw=black,fill=lightgray] (3,3)--(4,3)--(4,4)--(3,4)--(3,3);
		\filldraw[draw=black,fill=lightgray] (3,1)--(4,1)--(4,2)--(3,2)--(3,1);
		\node at (3,-.65) {\large $D(315264)$};
		\end{tikzpicture}
		\quad
		\begin{tikzpicture}[scale=.625]
		\draw (0,0)--(6,0)--(6,6)--(0,6)--(0,0);
		\draw (6,5.5) -- (2.5,5.5) node {$\bullet$} -- (2.5,0);
		\draw (6,4.5) -- (0.5,4.5) node {$\bullet$} -- (0.5,0);
		\draw (6,3.5) -- (4.5,3.5) node {$\bullet$} -- (4.5,0);
		\draw (6,2.5) -- (5.5,2.5) node {$\bullet$} -- (5.5,0);
		\draw (6,1.5) -- (1.5,1.5) node {$\bullet$} -- (1.5,0);
		\draw (6,0.5) -- (3.5,0.5) node {$\bullet$} -- (3.5,0);
		
		\filldraw[draw=black,fill=lightgray] (0,5)--(1,5)--(1,6)--(0,6)--(0,5);
		\filldraw[draw=black,fill=lightgray] (1,5)--(2,5)--(2,6)--(1,6)--(1,5);
		\filldraw[draw=black,fill=lightgray] (1,3)--(1,4)--(2,4)--(2,3)--(1,3);
		\filldraw[draw=black,fill=lightgray] (3,3)--(4,3)--(4,4)--(3,4)--(3,3);
		\filldraw[draw=black,fill=lightgray] (3,2)--(4,2)--(4,3)--(3,3)--(3,2);
		\filldraw[draw=black,fill=lightgray] (1,2)--(2,2)--(2,3)--(1,3)--(1,2);
		
		\node at (3,-.65) {\large $D(315624)$};
		\end{tikzpicture}
		\quad
		\begin{tikzpicture}[scale=.625]
		\draw (0,0)--(6,0)--(6,6)--(0,6)--(0,0);
		\draw (6,5.5) -- (2.5,5.5) node {$\bullet$} -- (2.5,0);
		\draw (6,4.5) -- (0.5,4.5) node {$\bullet$} -- (0.5,0);
		\draw (6,3.5) -- (4.5,3.5) node {$\bullet$} -- (4.5,0);
		\draw (6,2.5) -- (5.5,2.5) node {$\bullet$} -- (5.5,0);
		\draw (6,1.5) -- (3.5,1.5) node {$\bullet$} -- (3.5,0);
		\draw (6,0.5) -- (1.5,0.5) node {$\bullet$} -- (1.5,0);
		
		\filldraw[draw=black,fill=lightgray] (0,5)--(1,5)--(1,6)--(0,6)--(0,5);
		\filldraw[draw=black,fill=lightgray] (1,5)--(2,5)--(2,6)--(1,6)--(1,5);
		\filldraw[draw=black,fill=lightgray] (1,3)--(1,4)--(2,4)--(2,3)--(1,3);
		\filldraw[draw=black,fill=lightgray] (3,3)--(4,3)--(4,4)--(3,4)--(3,3);
		\filldraw[draw=black,fill=lightgray] (3,2)--(4,2)--(4,3)--(3,3)--(3,2);
		\filldraw[draw=black,fill=lightgray] (1,2)--(2,2)--(2,3)--(1,3)--(1,2);
		\filldraw[draw=black,fill=lightgray] (1,1)--(2,1)--(2,2)--(1,2)--(1,1);
		\node at (3,-.65) {\large $D(315642)$};
		\end{tikzpicture}
		\caption{The Rothe diagrams of the twelve multiplicitous patterns.}
		\label{fig:badpatterns}
	\end{figure}

	\begin{theorem}\label{thm:mult-pattern}
		If $w\in S_n$ is a permutation such that $D(w)$ does not contain any instance of configuration $\mathrm{A}$, $\mathrm{B}$, or $\mathrm{B}'$,
		then $w$ is multiplicity-free.
	\end{theorem}
	Theorem~\ref{thm:011} will also imply the converse of this theorem.

	\begin{proof}
		We prove the contrapositive.
		Assume $w$ is not multiplicity-free and let $(\bm{i},\bm{m})$ be the orthodontic sequence of $w$. Then, we can find entries $i_{p_1}=i_{p_2}$ of $\bm{i}$ with $p_1<p_2$ such that either $\mathcal{I}_w(p_1)\neq \mathcal{I}_w(p_2)$, or $\mathcal{I}_w(p_1)=\mathcal{I}_w(p_2)$ with $|\mathcal{I}_w(p_1)|>1$. We show that $D(w)$ must contain at least one instance of configuration $\mathrm{A}$, $\mathrm{B}$, or $\mathrm{B}'$.\\
		
		
		\noindent\emph{Case 1:} Assume that $\mathcal{I}_w(p_1)\nsubseteq\mathcal{I}_w(p_2)$ and $\mathcal{I}_w(p_2)\nsubseteq\mathcal{I}_w(p_1)$.		
		Take $c_1\in \mathcal{I}_w(p_1)\backslash \mathcal{I}_w(p_2)$ and $c_2\in \mathcal{I}_w(p_2)\backslash \mathcal{I}_w(p_1)$.
		We show that columns $c_1$ and $c_2$ of $D(w)$ contain an instance of configuration $\mathrm{A}$. 
		
		In step $p_1$ of the orthodontia on $D(w)$, a box in column $c_1$ is moved (by the missing tooth $i_{p_1}$) to row $i_{p_1}$. Let this box originally be in row $r_1$ of $D(w)$. Analogously, let the box in column $c_2$ moved to row $i_{p_2}$ in step $p_2$ of the orthodontia (by the missing tooth $i_{p_2}$) originally be in row $r_2$ of $D(w)$.
		Observe that $r_1<r_2$. 
		If $c_2<c_1$, then the northwest property would imply that $(r_1,c_2)\in D(w)$, contradicting that $c_2\notin \mathcal{I}_w(p_1)$. Thus $c_1<c_2$. Since $c_2\notin \mathcal{I}_w(p_1)$, $(r_1,c_2)\notin D(w)$. 
		Lastly, since the box $(r_1,c_1)$ is moved by the orthodontia, there is some box $(r_3,c_1)\notin D(w)$ with $r_3<r_1$. Consequently, $w_{r_3}<c_1$. Thus, $(r_1,c_1,r_2,c_2,r_3)$ is an instance of configuration $\mathrm{A}$.\\
		
		
		\noindent\emph{Case 2:}
		Assume $\mathcal{I}_w(p_2)$ is a proper subset of $\mathcal{I}_w(p_1)$.
		Let $c_1=\max(\mathcal{I}_w(p_2))$
		and $c_2=\min(\mathcal{I}_w(p_1)\backslash\mathcal{I}_w(p_2))$.
		Let the box in column $c_1$ moved to row $i_{p_1}=i_{p_2}$ in step $p_1$ (resp. $p_2$) of the orthodontia originally be in row $r_1$ (resp. $r_2$) of $D(w)$. Observe that $r_1<r_2$.
		
		Assume first that $c_1<c_2$. Since $c_1\in\mathcal{I}_w(p_1)\cap \mathcal{I}_w(p_2)$, the boxes $(r_1,c_1)$ and $(r_2,c_2)$ both move weakly above row $i_{p_1}$ in the orthodontia. Then,
		we can find indices $r_3,r_4$ with $r_4<r_3<r_1$ such that $(r_3,c_1),(r_4,c_1)\notin D(w)$. Hence, $w_{r_3}<c_1$ and $w_{r_4}<c_1$, so $(r_1,c_1,r_2,c_2,r_3,r_4)$ is an instance of configuration $\mathrm{B}'$.
		
		Otherwise $c_1>c_2$. Since the box $(r_1,c_2)$ is moved by the orthodontia, we can find $r_3<r_1$ with $(r_3,c_2)\notin D(w)$. Then $w_{r_3}<c_2$. As we are assuming $c_2<c_1$, $(r_3,c_1)\notin D(w)$ also. Since the boxes $(r_1,c_1)$ and $(r_2,c_1)$ in $D(w)$ are moved weakly above row $i_{p_1}$ by the orthodontia, we can find some $r_4<r_1$ with $r_4\neq r_3$ such that $(r_4,c_1)\notin D(w)$. Then, $w_{r_4}<c_1$, so $(r_1,c_2,r_2,c_1,r_3,r_4)$ is an instance of configuration $\mathrm{B}$.\\
		
		
		\noindent\emph{Case 3:} 
		Assume $\mathcal{I}_w(p_1)$ is a proper subset of $\mathcal{I}_w(p_2)$.
		This case is handled similarly to Case 2. 
		Let $c_1=\max(\mathcal{I}_w(p_1))$ and $c_2=\min(\mathcal{I}_w(p_2)\backslash\mathcal{I}_w(p_1))$.
		Let the box in column $c_1$ moved to row $i_{p_1}=i_{p_2}$ in step $p_1$ (resp. $p_2$) of the orthodontia originally be in row $r_1$ (resp. $r_2$) of $D(w)$. Observe that $r_1<r_2$.
		
		Assume $c_1<c_2$. Since the boxes $(r_1,c_1)$ and $(r_2,c_1)$ of $D(w)$ are moved weakly above row $i_{p_1}$ by the orthodontia, we can find indices $r_3,r_4$ with $r_4<r_3<r_1$ such that $(r_3,c_1),(r_4,c_1)\notin D(w)$. Then, $w_{r_3}<c_1$ and $w_{r_4}<c_1$. 
		Since $c_2\notin \mathcal{I}_w(p_1)$,
		$(r_1,c_2)\notin D(w)$. Then, $(r_1,c_1,r_2,c_2,r_3)$ is an instance of configuration $\mathrm{A}$.
		
		Otherwise $c_1>c_2$.
		As $c_2\notin \mathcal{I}_w(p_1)$, 
		$(r_1,c_2)\notin D(w)$. Since $(r_2,c_2),(r_1,c_1)\in D(w)$, this is a contradiction of the northwest property of $D(w)$.\\
		
		
		\noindent\emph{Case 4:} 
		Assume $\mathcal{I}_w(p_1)=\mathcal{I}_w(p_2)$ is not a singleton. 
		Let $c_1,c_2\in \mathcal{I}_w(p_1)$ with $c_1<c_2$.
		Let the box in column $c_1$ moved to row $i_{p_1}=i_{p_2}$ in step $p_1$ (resp. $p_2$) of the orthodontia originally be in row $r_1$ (resp. $r_2$) of $D(w)$. Observe that $r_1<r_2$. Since the boxes $(r_1,c_1)$ and $(r_2,c_1)$ in $D(w)$ are moved weakly above row $i_{p_1}$ by the orthodontia, we can find indices $r_3,r_4$ with $r_4<r_3<r_1$ such that $(r_3,c_1),(r_4,c_1)\notin D(w)$. Then, $w_{r_3}<c_1$ and $w_{r_4}<c_1$. Thus, $(r_1,c_1,r_2,c_2,r_3,r_4)$ is an instance of configuration $\mathrm{B}'$.
	\end{proof}

	We now relate multiplicity-freeness to pattern avoidance of permutations. We begin by clarifying our pattern avoidance terminology. 
	A \emph{pattern} $\sigma$ of \emph{length}~$n$ is a permutation in~$S_n$.
	The length $n$ is a crucial part of the data of a pattern; 
	we make no identifications between patterns of different lengths, unlike what is usual when handling permutations in the Schubert calculus.
	A  permutation $w$ \emph{contains} $\sigma$ if $w$ has $n$ entries $w_{j_1},\ldots, w_{j_n}$ with $j_1<j_2<\cdots<j_n$ that are in the same relative order as $\sigma_1,\sigma_2,\ldots,\sigma_n$. In this case, the indices $j_1<j_2<\cdots<j_n$ are called a \emph{realization} of $\sigma$ in $w$. We say that $w$ \emph{avoids} the pattern $\sigma$ if $w$ does not contain $\sigma$. 
	To illustrate the dependence of these definitions on $n$, note that $w=154623$ contains the pattern 132, but not the pattern 132456.
	
	The following easy lemma gives a diagrammatic interpretation of pattern avoidance.
	\begin{lemma}
		\label{lem:patterndiagram}
		Let $w\in S_n$ be a permutation and $\sigma$ a pattern of length~$m$ contained in $w$. Choose a realization 
		$j_1<j_2<\cdots<j_{m}$ of $\sigma$ in $w$. 
		Then $D(\sigma)$ is obtained from $D(w)$ by deleting the rows $[n]\backslash\{j_1,\ldots,j_{m} \}$ and the columns $[n]\backslash\{w_{j_1},\ldots,w_{j_{m}} \}$, 
		and reindexing the remaining rows and columns by~$[m]$, preserving their order.
	\end{lemma}
		
	\begin{definition}
		The \emph{multiplicitous patterns} are those in the set
		\[\mathrm{MPatt}=\{12543, 13254, 13524, 13542, 21543, 125364, 125634, 215364, 215634, 315264, 315624, 315642\}.\]
	\end{definition}
	
	\begin{theorem}
		\label{thm:badpatterns}
		Let $w\in S_n$.  Then $D(w)$ does not contain any instance of configuration $\mathrm{A}$, $\mathrm{B}$, or $\mathrm{B}'$ if and only if $w$ avoids all of the multiplicitous patterns.
	\end{theorem}
	
	\begin{proof}
		It is easy to check (see Figure~\ref{fig:badpatterns}) that each of the twelve multiplicitous patterns contains an instance of configuration $\mathrm{A}$, $\mathrm{B}$, or $\mathrm{B}'$. Lemma~\ref{lem:patterndiagram} implies that if $w$ contains $\sigma\in\mathrm{MPatt}$, then deleting some rows and columns from $D(w)$ yields $D(\sigma)$. Since $D(\sigma)$ contains at least one instance of configuration $\mathrm{A}$, $\mathrm{B}$, or $\mathrm{B}'$, so does $D(w)$.
		
		Conversely, assume $D(w)$ contains at least one instance of configuration $\mathrm{A}$, $\mathrm{B}$, or $\mathrm{B}'$. We must show that $w$ contains some multiplicitous pattern. 
		Let $\tau^1,\tau^2,\ldots,\tau^n$ be the $n$ patterns of length~$n-1$ contained in $w$; 
		say $\tau^j$ is realized in $w$ by forgetting $w_j$. 
		Without loss of generality, we may assume none of $D(\tau^1),\ldots,D(\tau^n)$ contain an instance of configuration $\mathrm{A}$, $\mathrm{B}$, or $\mathrm{B}'$: if $D(\tau^j)$ does contain an instance of one of these configurations, replace $w$ by $\tau^j$ and iterate. 
		
		For each $j$, $D(\tau^j)$ is obtained from $D(w)$ by deleting row $j$ and column $w_j$. 
		Since $D(\tau^j)$ does not contain any instance of any of our three configurations, 
		each cross $\{(j,q) \mid (j,q)\in D(w) \}\cup \{(p,w_j) \mid (p,w_j)\in D(w)\}$ intersects each instance of every configuration contained in $D(w)$. However, an instance of configuration $\mathrm{A}$ involves only three rows and two columns, and an instance of $\mathrm{B}$ or $\mathrm{B}'$ involves only four rows and two columns. Thus, it must be that $w\in S_n$ for some $n\leq 6$. 
		It can be checked by exhaustion that the only permutations in $S_n$ with $n\leq 6$ that are minimal (with respect to pattern avoidance) among those whose Rothe diagrams contain an instance of configuration $\mathrm{A}$, $\mathrm{B}$, or $\mathrm{B}'$ are the twelve multiplicitous patterns.
	\end{proof}
	
	We are now ready to state our full characterization of zero-one Schubert polynomials,
	and most of the elements of the proof are at hand.
	
	
	\begin{theorem} \label{thm:011}
	The following are equivalent:
		\begin{itemize}
			\item[(i)] The Schubert polynomial $\mathfrak{S}_w$ is zero-one.
			\item[(ii)] The permutation $w$ is multiplicity-free,
			\item[(iii)] The Rothe diagram $D(w)$ does not contain any instance of configuration $\mathrm{A}$, $\mathrm{B}$, or $\mathrm{B}'$,
			\item[(iv)] The permutation $w$ avoids the multiplicitous patterns, namely $12543$, $13254$, $13524$, $13542$, $21543$, $125364$, $125634$, $215364$, $215634$, $315264$, $315624$, and $315642$.
		\end{itemize}
	\end{theorem}
	\begin{proof} 
		Theorem~\ref{thm:multfree} shows $(ii)\Rightarrow(i)$. Theorem~\ref{thm:mult-pattern} shows $(iii)\Rightarrow(ii)$. Theorem~\ref{thm:badpatterns} shows $(iii)\Leftrightarrow(iv)$. The implication $(i)\Rightarrow(iv)$ will follow immediately from Corollary~\ref{cor:pattern}, since the Schubert polynomials associated to the permutations $12543$, $13254$, $13524$, $13542$, $21543$, $125364$, $125634$, $215364$, $215634$, $315264$, $315624$, and $315642$ each have a coefficient equal to 2. We prove Corollary~\ref{cor:pattern} in the next section.
	\end{proof}

\section{A coefficient-wise inequality for dual characters of flagged Weyl modules of diagrams}
	\label{sec:trans}
	The aim of this section is to prove a generalization of Theorem~\ref{thm:pattern}, namely, Theorem~\ref{thm:pattern2}.		
We now explain the necessary background and terminology for Theorem~\ref{thm:pattern2} and its proof.


	Let $G=\mathrm{GL}(n,\mathbb{C})$ be the group of $n\times n$ invertible matrices over $\mathbb{C}$ and $B$ be the subgroup of $G$ consisting of the $n\times n$ upper-triangular matrices. The flagged Weyl module is a representation $\mathcal{M}_D$ of $B$ associated to a diagram $D$. The dual character of $\mathcal{M}_D$ has been shown in certain cases to be a Schubert polynomial \cite{KP} or a key polynomial \cite{flaggedLRrule}. We will use the construction of $\mathcal{M}_D$ in terms of determinants given in \cite{magyar}.
	
	Denote by $Y$ the $n\times n$ matrix with indeterminates $y_{ij}$ in the upper-triangular positions $i\leq j$ and zeros elsewhere. Let $\mathbb{C}[Y]$ be the polynomial ring in the indeterminates $\{y_{ij}\}_{i\leq j}$. Note that $B$ acts on $\mathbb{C}[Y]$ on the right via left translation: if $f(Y)\in \mathbb{C}[Y]$, then a matrix $b\in B$ acts on $f$ by $f(Y)\cdot b=f(b^{-1}Y)$. For any $R,S\subseteq [n]$, let $Y_S^R$ be the submatrix of $Y$ obtained by restricting to rows $R$ and columns $S$.
	
	For $R,S\subseteq [n]$, we say $R\leq S$ if $\#R=\#S$ and the $k$\/th least element of $R$ does not exceed the $k$\/th least element of $S$ for each $k$. For any diagrams $C=(C_1,\ldots, C_n)$ and $D=(D_1,\ldots, D_n)$, we say $C\leq D$ if $C_j\leq D_j$ for all $j\in[n]$.

	\begin{definition}
		For a diagram $D=(D_1,\ldots, D_n)$, the \emph{flagged Weyl module} $\mathcal{M}_D$ is defined by
		\[\mathcal{M}_D=\mathrm{Span}_\mathbb{C}\left\{\prod_{j=1}^{n}\det\left(Y_{D_j}^{C_j}\right)\ \middle|\   C\leq D \right\}. \]
		$\mathcal{M}_D$ is a $B$-module with the action inherited from the action of $B$ on $\mathbb{C}[Y]$. 
	\end{definition}
	Note that since $Y$ is upper-triangular, the condition $C\leq D$ is technically unnecessary since $\det\left(Y_{D_j}^{C_j}\right)=0$ unless $C_j\leq D_j$. Conversely, if $C_j\leq D_j$, then $\det\left(Y_{D_j}^{C_j}\right)\neq 0$. 
	
	For any $B$-module $N$, the \emph{character} of $N$ is defined by $\mathrm{char}(N)(x_1,\ldots,x_n)=\mathrm{tr}\left(X:N\to N\right)$ where $X$ is the diagonal matrix $\mathrm{diag}(x_1,x_2,\ldots,x_n)$ with diagonal entries $x_1,\ldots,x_n$, and $X$ is viewed as a linear map from $N$ to $N$ via the $B$-action. Define the \emph{dual character} of $N$ to be the character of the dual module $N^*$:
	\begin{align*}
		\mathrm{char}^*(N)(x_1,\ldots,x_n)&=\mathrm{tr}\left(X:N^*\to N^*\right) \\
		&=\mathrm{char}(N)(x_1^{-1},\ldots,x_n^{-1}).
	\end{align*}

	A special case of dual characters of flagged Weyl modules of diagrams are Schubert polynomials:
	
	\begin{theorem}[\cite{KP}]
		\label{thm:kp}
		Let $w$ be a permutation, $D(w)$ be the Rothe diagram of $w$, and $\mathcal{M}_{D(w)}$ be the associated flagged Weyl module. Then, 
		\[\mathfrak{S}_w = \mathrm{char}^*\mathcal{M}_{D(w)}. \]
	\end{theorem}

	Another special family  of dual characters of flagged Weyl modules of diagrams, for so-called skyline diagrams of compositions, are key polynomials \cite{keypolynomials}.
	
	\begin{definition}
		For a diagram $D\subseteq [n]\times [n]$, let $\chi_D=\chi_D(x_1,\ldots,x_n)$ be the dual character 
		\[\chi_D=\mathrm{char}^*\mathcal{M}_D. \] 
	\end{definition}

	We now work towards proving Theorem~\ref{thm:pattern2}. 
	We start by reviewing some material from \cite{FMS} for the reader's convenience.
	We then derive several lemmas that simplify the proof of Theorem~\ref{thm:pattern2}.

	\begin{theorem}[cf. {\cite[Theorem 7]{FMS}}]
		For any diagram $D\subseteq [n]\times [n]$, the monomials appearing in $\chi_D$ are exactly 
		\[\left\{\prod_{j=1}^{n}\prod_{i\in C_j}x_i\ \middle|\   C\leq D \right\}.\]
	\end{theorem}
	\begin{proof}
		(Following that of \cite[Theorem 7]{FMS}) Denote by $X$ the diagonal matrix $\mathrm{diag}(x_1,x_2,\ldots,x_n)$. First, note that $y_{ij}$ is an eigenvector of $X$ with eigenvalue $x_i^{-1}$. Take a diagram $C=(C_1,\ldots,C_n)$ with $C\leq D$. Then, the element $\prod_{j=1}^{n}\det\left(Y_{D_j}^{C_j}\right)$ is an eigenvector of $X$ with eigenvalue $\prod_{j=1}^{n}\prod_{i\in C_j}x_i^{-1}$. Since $\mathcal{M}_D$ is spanned by elements $\prod_{j=1}^{n}\det\left(Y_{D_j}^{C_j}\right)$ and each is an eigenvector of $X$, the monomials appearing in the dual character $\chi_D$ are exactly 
		$\left\{\prod_{j=1}^{n}\prod_{i\in C_j}x_i\ \middle|\   C\leq D \right\}$.
	\end{proof}

	\begin{corollary}
		\label{cor:fms}
		Let $D\subseteq [n]\times [n]$ be a diagram. Fix any diagram $C^{(1)}\leq D$ and set \[\bm{m}=\prod_{j=1}^{n}\prod_{i\in C^{(1)}_j}x_i.\] 
		Let $C^{(1)}, \ldots, C^{(r)}$ be all the diagrams $C$ such that $C\leq D$ and $\prod_{j=1}^{n}\prod_{i\in C_j}x_i=\bm{m}$. Then, the coefficient of $\bm{m}$ in $\chi_D$ is equal to 
		\[\dim \left(\mathrm{Span}_\mathbb{C}\left\{\prod_{j=1}^{n}\det\left(Y_{D_j}^{C^{(i)}_j}\right) \ \middle|\  i\in [r] \right\}\right).\] 
	\end{corollary}

	\begin{proof} 
		The coefficient of $\bm{m}$ in $\chi_D$ equals the dimension of the eigenspace of $\bm{m}^{-1}$ in $\mathcal{M}_D$ ($\bm{m}^{-1}$ occurs here instead of $\bm{m}$ since $\chi_D$ is the dual character of $\mathcal{M}_D$). This eigenspace equals 
		\[\mathrm{Span}_{\mathbb{C}}\left\{\prod_{j=1}^{n}\det\left(Y_{D_j}^{C^{(i)}_j}\right) \ \middle|\  i\in [r] \right\},\] so the result follows.
	\end{proof}
	  
	The understanding of the coefficients of the monomials of $\chi_D$ given in Corollary~\ref{cor:fms}  is  key to our proof of Theorem~\ref{thm:pattern2}. We set up some notation now. 
 
 	
	Given diagrams $C,D\subseteq [n]\times[n]$ and $k,l\in [n]$, let $\hc$ and  $\hd$ denote the diagrams obtained from $C$ and $D$ by removing any boxes in row $k$ or column $l$. 
	Fix a diagram $D$. For each diagram $\hc$, let 
	\[{\hc}_{\rm aug}=\hc \cup \{(k,i) \mid (k,i)\in D\} \cup \{(i,l) \mid (i,l) \in D \}\subseteq [n]\times [n].\]

	The following lemma is immediate and its proof is left to the reader.  
	\begin{lemma} 
		\label{lem:c}
		Let $C, D \subseteq [n]\times[n]$ be diagrams and $k,l\in [n]$. If $\hc\leq \hd$, then ${\hc}_{\rm aug}\leq D$. 
		In particular, every diagram $C'\leq \hd$ with no boxes in row $k$
		can be obtained from some diagram $C\leq D$ by removing any boxes in row $k$ or column $l$ from $C$.
	\end{lemma}
 
	The following result is our key lemma.
	For a polynomial $f\in\mathbb{Z}[x_1,\ldots,x_n]$ and a monomial $\bm{m}$, let $[\bm{m}]f$ denote the coefficient of $\bm{m}$ in $f$.
	\begin{lemma} 
		\label{lem:keylemma}
		Fix a diagram $D$ and $k,l\in [n]$. Let $\{\widehat{C}^{(i)}\}_{i\in [m]}$ be a set of diagrams with $\widehat{C}^{(i)}\leq \widehat{D}$ for each $i$, and denote $\widehat{C}^{(i)}_{\mathrm{aug}}$ by $C^{(i)}$ for $i\in[m]$. If the polynomials $\displaystyle\left\{\prod_{j\in [n]}\det\left(Y_{D_j}^{{C^{(i)}_j}}\right)\right\}_{i \in [m]}$ are linearly dependent, then so are the polynomials $\displaystyle\left\{\prod_{j\in [n]\backslash \{l\}}\det\left(Y_{\hd_j}^{\widehat{C}^{(i)}_j}\right)\right\}_{i \in [m]}$. 
	\end{lemma}
	\begin{proof}
		We are given that 
		\begin{align}\label{eq:proof-1} \sum_{i\in [m]}{c_{i}}\prod_{j\in [n]}\det\left(Y_{D_j}^{C^{(i)}_j}\right)=0\end{align} for some constants $(c_{i})_{i \in [m]}\in \mathbb{C}^m$ not all zero. Since ${C^{(i)}}=\widehat{C}^{(i)}_{\rm aug}$ for $\widehat{C}^{(i)}\leq \hd$ we have that  ${C_l^{(i)}}=D_l$ for every $i \in [m]$. Thus, \eqref{eq:proof-1} can be rewritten as
		
		\begin{align*}\det\left(Y_{D_l}^{{D_l}}\right)\left(\sum_{i\in [m]}{c_{i}}\prod_{j\in [n]\backslash \{l\}}\det\left(Y_{D_j}^{C^{(i)}_j}\right)\right)=0.\end{align*}

		However, since $\det\left(Y_{D_l}^{{D_l}}\right)\neq 0$, we conclude that 
		
		\begin{align}\label{eq:proof-2} \sum_{i\in [m]}{c_{i}}\prod_{j\in [n]\backslash \{l\}}\det\left(Y_{D_j}^{C^{(i)}_j}\right)=0.\end{align}

		First consider the case that the only boxes of $D$ in row $k$ or column $l$ are those in $D_l$. If this is the case then 
		
		\begin{align*}
			\prod_{j\in [n]\backslash \{l\}}\det\left(Y_{\hd_j}^{{\widehat{C}^{(i)}}_j}\right)=\prod_{j\in [n]\backslash \{l\}}\det\left(Y_{D_j}^{C^{(i)}_j}\right)
		\end{align*} 
		for each $i \in [m]$. 
		Therefore,
		
		\begin{align}\label{eq:proof-3} \sum_{i\in [m]}{c_{i}}\prod_{j\in [n]\backslash \{l\}}\det\left(Y_{\hd_j}^{{\widehat{C}^{(i)}}_j}\right)=\sum_{i\in [m]}{c_{i}}\prod_{j\in [n]\backslash \{l\}}\det\left(Y_{D_j}^{C^{(i)}_j}\right).\end{align} 
		
		Combining \eqref{eq:proof-2} and \eqref{eq:proof-3} we obtain that  the polynomials $\left\{\prod_{j\in [n]\backslash \{l\}}\det\left(Y_{\hd_j}^{\widehat{C}^{(i)}_j}\right)\right\}_{i \in [m]}$ are linearly dependent, as desired.

		Now, suppose that there are boxes  of $D$ in row $k$ that are not in  $D_l$. Let $j_1< \ldots< j_p$ be all indices $j \neq l$ such that $D_j=\hd_j \cup \{k\}$. 
		Then also $C^{(i)}_{j_q}=\widehat{C}^{(i)}_{j_q} \cup \{k\}$ for each $i \in [m]$ and $q\in [p]$.
		View the left-hand side of \eqref{eq:proof-2} as a polynomial in $y_{kk}$. Then, \eqref{eq:proof-2} implies that the coefficient of $y_{kk}^p$ is $0$:
		
		\begin{align}\label{eq:proof-4} [y_{kk}^p]\sum_{i\in [m]}{c_{i}}\prod_{j\in [n]\backslash \{l\}}\det\left(Y_{D_j}^{C^{(i)}_j}\right)=0.\end{align} 
		
		On the other hand, the determinants $\det\left(Y_{D_j}^{C^{(i)}_j}\right)$ involve $y_{kk}$ exactly when $j=j_q$ for some $q\in[p]$. In this case, applying Laplace expansion along the row of $Y_{D_{j_q}}^{C^{(i)}_{j_q}}$ containing $y_{kk}$ implies
		
		\begin{align*}
			[y_{kk}]\det\left(Y_{D_{j_q}}^{C^{(i)}_{j_q}}\right) = \xi_{i,q}\det\left(Y_{\widehat{D}_{j_q}}^{\widehat{C}^{(i)}_{j_q}}\right)
		\end{align*}
		with $\xi_{i,q}\in\{1,-1 \}$. Thus,
		
		
		\begin{align*}
			[y_{kk}^p]\sum_{i\in [m]}{c_{i}}\prod_{j\in [n]\backslash \{l\}}\det\left(Y_{D_j}^{C^{(i)}_j}\right)
			&=\sum_{i\in [m]}{c_{i}}\left([y_{kk}^p]\prod_{j\in [n]\backslash \{l\}}\det\left(Y_{D_j}^{C^{(i)}_j}\right)\right)\\
			&=\sum_{i\in [m]}{c_{i}}\left(\prod_{j\in [n]\backslash \{l,j_1,\ldots,j_p\}}\det\left(Y_{D_j}^{C^{(i)}_j}\right)\right)\left([y_{kk}^p]\prod_{q\in [p]}\det\left(Y_{D_{j_q}}^{C^{(i)}_{j_q}}\right)\right)\\
			&=\sum_{i\in [m]}{c_{i}}\left(\prod_{j\in [n]\backslash \{l,j_1,\ldots,j_p\}}\det\left(Y_{D_j}^{C^{(i)}_j}\right)\right)\left(\prod_{q\in [p]}[y_{kk}]\det\left(Y_{D_{j_q}}^{C^{(i)}_{j_q}}\right)\right)\\
			&=\sum_{i\in [m]}{c_{i}}\left(\prod_{j\in [n]\backslash \{l,j_1,\ldots,j_p\}}\det\left(Y_{D_j}^{C^{(i)}_j}\right)\right)\left(\prod_{q\in [p]}\xi_{i,q}\det\left(Y_{\widehat{D}_{j_q}}^{\widehat{C}^{(i)}_{j_q}}\right)\right).
		\end{align*} 
		Since $\widehat{C}_j^{(i)}=C_j^{(i)}$ and $\widehat{D}_j=D_j$ whenever $j\neq l,j_1,\ldots,j_p$, we obtain
		\begin{align*}
			[y_{kk}^p]\sum_{i\in [m]}{c_{i}}\prod_{j\in [n]\backslash \{l\}}\det\left(Y_{D_j}^{C^{(i)}_j}\right)
			&=\sum_{i\in [m]}{c_{i}}\left(\prod_{j\in [n]\backslash \{l,j_1,\ldots,j_p\}}\det\left(Y_{\widehat{D}_j}^{\widehat{C}^{(i)}_j}\right)\right)\left(\prod_{q\in [p]}\xi_{i,q}\det\left(Y_{\widehat{D}_{j_q}}^{\widehat{C}^{(i)}_{j_q}}\right)\right)\\
			&=\sum_{i\in [m]}{c_{i}}\left(\prod_{q\in [p]}\xi_{i,q}\right)
			\left(\prod_{j\in [n]\backslash \{l\}}\det\left(Y_{\widehat{D}_j}^{\widehat{C}^{(i)}_j}\right)\right).\\
		\end{align*}
		
		Setting $c_i'=c_i\prod_{q\in [p]}\xi_{i,q}$ and applying (\ref{eq:proof-4}) yields the dependence relation
		\begin{align*}
			\sum_{i\in [m]}c_{i}'
			\prod_{j\in [n]\backslash \{l\}}\det\left(Y_{\widehat{D}_j}^{\widehat{C}^{(i)}_j}\right)=0.
		\end{align*}
		This implies the polynomials $\displaystyle\left\{\prod_{j\in [n]\backslash \{l\}}\det\left(Y_{\hd_j}^{\widehat{C}^{(i)}_j}\right)\right\}_{i \in [m]}$ are linearly dependent, concluding the proof.
		
		
	\end{proof}
	
	We now state and prove Theorem~\ref{thm:pattern} and its generalization Theorem~\ref{thm:pattern2}.
	\begin{theorem}
		\label{thm:pattern2}
		Fix a diagram $D\subseteq [n]\times [n]$ and let $\hd$ be the diagram obtained from $D$ by removing any boxes in row $k$ or column $l$. Then 
		\begin{align*}
		\chi_D(x_1, \ldots, x_n)=M(x_1, \ldots, x_n) \chi_{\hd}(x_{1}, \ldots,x_{k-1},0,x_{k+1},\ldots, x_{n})+F(x_1, \ldots, x_n),
		\end{align*}
		where $F(x_1,\ldots,x_n) \in \Z_{\geq 0}[x_1, \ldots, x_n]$ and  
		\[M(x_1, \ldots, x_n) = \left(\prod_{(k,i)\in D}{x_k}\right)\left(\prod_{(i,l)\in D}{x_i} \right).\]
	\end{theorem}
	\begin{proof}
		Let $M=M(x_1,\ldots,x_n)$. 
		We must show that $[M\bm{m}]\chi_D\geq [\bm{m}]\chi_{\hd}$  for each monomial $\bm{m}$ of $\chi_{\hd}$ not divisible by $x_k$.
		Let $C^{(1)}, \ldots, C^{(r)}$ be all the diagrams $C$ such that $C\leq D$ and $\prod_{j=1}^{n}\prod_{i\in C_j}x_i=M\bm{m}$. By Corollary~\ref{cor:fms}, 
		\[[M\bm{m}]\chi_D=\dim\left(\mathrm{Span}_{\mathbb{C}} \left\{\prod_{j=1}^{n}\det\left(Y_{D_j}^{C^{(i)}_j}\right) \ \middle|\  i\in [r] \right\}\right).\]
		Let $1,2,\ldots,q$ be the indices of the distinct diagrams among $\widehat{C}^{(1)}, \ldots, \widehat{C}^{(r)}$ such that $\widehat{C}^{(j)}\leq \widehat{D}$ for $j\in [q]$. 
		By Lemma~\ref{lem:c}, $\widehat{C}^{(1)}, \ldots, \widehat{C}^{(q)}$ are all the diagrams $C$ such that $C\leq \widehat{D}$ and $\prod_{j=1}^{n}\prod_{i\in C_j}x_i=\bm{m}$,
		as no diagram with this dual eigenvalue can have a box in row~$k$.
		So Corollary~\ref{cor:fms} implies that 
		\[[\bm{m}]\chi_{\widehat{D}}=\dim\left(\mathrm{Span}_{\mathbb{C}} \left\{\prod_{j=1}^{n}\det\left(Y_{\widehat{D}_j}^{\widehat{C}^{(i)}_j}\right) \ \middle|\  i\in [q] \right\}\right).\]
		
		Finally, Lemma~\ref{lem:keylemma} implies that 
		\[
			\dim\left(\mathrm{Span}_{\mathbb{C}} \left\{\prod_{j=1}^{n}\det\left(Y_{D_j}^{C^{(i)}_j}\right) \ \middle|\  i\in [r] \right\}\right)\geq 
			\dim\left(\mathrm{Span}_{\mathbb{C}} \left\{\prod_{j=1}^{n}\det\left(Y_{\widehat{D}_j}^{\widehat{C}^{(i)}_j}\right) \ \middle|\  i\in [q] \right\}\right),
		\]
		so $[M\bm{m}]\chi_D\geq [\bm{m}]\chi_{\hd}$ for each monomial $\bm{m}$ of $\chi_{\hd}$ not divisible by $x_k$; that is 
		\[\chi_D(x_1, \ldots, x_n)-M\chi_{\hd}(x_{1}, \ldots,x_{k-1},0,x_{k+1},\ldots, x_{n})\in \mathbb{Z}_{\geq 0}[x_1,\ldots,x_n].\]
	\end{proof}
	
	\begin{namedtheorem}[\ref{thm:pattern}]  
		Fix  $w \in S_n$ and let $\sigma \in S_{n-1}$ be the pattern with Rothe diagram $D(\sigma)$ obtained by removing row $k$ and column $w_k$ from $D(w)$. Then
		\begin{align*}
		\mathfrak{S}_{w}(x_1, \ldots, x_n)=M(x_1, \ldots, x_n) \mathfrak{S}_{\sigma}(x_{1}, \ldots, \widehat{x_k}, \ldots, x_{n})+F(x_1, \ldots, x_n),
		\end{align*}  
		where $F\in \Z_{\geq 0}[x_1, \ldots, x_n]$ and  
		\[M(x_1, \ldots, x_n) = \left(\prod_{(k,i)\in D(w)}{x_k}\right)\left(\prod_{(i,w_k)\in D(w)}{x_i} \right).\]		
	\end{namedtheorem}
	\begin{proof}
		Specialize Theorem~\ref{thm:pattern2} to the case that $D$ is a Rothe diagram $D(w)$ and $l=w_k$. The dropping of $x_k$ is due to reindexing, since the entirety of row $k$ and column $w_k$ of $D(w)$ are removed from to obtain $D(\sigma)$, not just the boxes in row $k$ and column $w_k$.
	\end{proof}
 
	\begin{corollary} 
		\label{cor:pattern}  
		Fix  $w \in S_n$ and let $\sigma\in S_m$ be any pattern contained in $w$. If $k$ is a coefficient of a monomial in $\mathfrak{S}_\sigma$, then $\mathfrak{S}_w$ contains a monomial with coefficient at least $k$. 
	\end{corollary}
	\begin{proof}
		Immediate consequence of repeated applications of Theorem~\ref{thm:pattern}.
	\end{proof}

\section*{Acknowledgments} 
We are grateful to Sara Billey and Allen Knutson for many discussions about Schubert polynomials.
We thank the Institute for Advanced Study for providing a hospitable environment for our collaboration. Many thanks to Arthur Tanjaya for his careful reading.

\bibliographystyle{plain}
\bibliography{zeroonebiblio}
	
\end{document}